\documentclass[11pt, final]{article}
\usepackage{a4}

\usepackage{amsmath}
\usepackage{amssymb}
\usepackage{amsxtra}
\usepackage{amsthm}
\usepackage{bbm}			% \mathbb{1} funktioniert nicht, \mathbbm{1} schon
\usepackage{cite}

\usepackage[left=3.5cm,right=3.5cm,top=3cm,bottom=3cm]{geometry}

\usepackage[plainpages=false,pdfpagelabels]{hyperref} % erzeugt die Links in der PDF

\usepackage{showkeys}
\usepackage[latin1]{inputenc}       % F?r Umlaut und ? beim Schreiben
\usepackage[T1]{fontenc}        % f?r die Schriftart
\usepackage{lmodern}            % braucht man f?r die Skalierung der Schriften
\usepackage{graphicx}
\usepackage{subfigure}		% für mehrere Bilder in einer Abbildung
\usepackage{booktabs}		% für toprule, midrule und bottomrule in tabellen
\usepackage{multirow}
\usepackage{enumitem}		% für besondere Labels bei Aufzählungen
\usepackage{listings} 	% für code-Umgebungen

\usepackage{colortbl}
\definecolor{dunkelgrau}{rgb}{0.8,0.8,0.8}
\definecolor{hellgrau}{rgb}{0.9,0.9,0.9}

\newcommand{\R}{\ensuremath{\mathbb{R}}}

\newcommand{\bx}{\ensuremath{\overline{x}}}
\newcommand{\by}{\ensuremath{\overline{y}}}

\newcommand{\h}{\ensuremath{\mathcal{H}}}

\theoremstyle{plain}
\newtheorem{theorem}{Theorem}

\newtheorem{proposition}[theorem]{Proposition}

\theoremstyle{definition}
\newtheorem{remark}[theorem]{Remark}

\DeclareMathOperator*\dom{dom}%
\DeclareMathOperator*\ri{ri}%
\DeclareMathOperator*\argmin{arg\,min}%
\DeclareMathOperator*\cl{cl}%
\DeclareMathOperator*\id{id}%

\title{On the acceleration of the double smoothing technique for unconstrained convex optimization problems}

\author{Radu Ioan Bo\c t
\thanks {Faculty of Mathematics, Chemnitz University of Technology, D-09107 Chemnitz, Germany, e-mail: radu.bot@mathematik.tu-chemnitz.de. Research partially supported by DFG (German Research Foundation), project BO 2516/4-1.}
\and Christopher Hendrich
\thanks{Faculty of Mathematics, Chemnitz University of Technology, D-09107 Chemnitz, Germany, e-mail: christopher.hendrich@mathematik.tu-chemnitz.de.}
}
\date{\today}

\begin{document}
\maketitle

{\bf Abstract.} In this article we investigate the possibilities of accelerating the double smoothing technique when solving unconstrained nondifferentiable convex optimization problems. This approach relies on the regularization in two steps of the Fenchel dual problem associated to the problem to be solved into an optimization problem having a differentiable strongly convex objective function with Lipschitz continuous gradient. The doubly regularized dual problem is then solved via a fast gradient method. The aim of this paper is to show how do the properties of the functions in the objective of the primal problem influence the implementation of the double smoothing approach and its rate of convergence. The theoretical results are applied to linear inverse problems by making use of different regularization functionals.



\section{Introduction}\label{sectionIntro}

In this paper we are developing an efficient algorithm based on the double smoothing approach for solving unconstrained nondifferentiable optimization problems of the type
\begin{equation}\label{opt-problem:primal}
\hspace{-1.8cm}(P) \quad \quad \inf_{x \in \h}{\left\{f(x)+g(Ax)\right\}},
\end{equation}
where $\h$ is a Hilbert space, $f:\h \rightarrow \overline{\mathbb{R}}$ and $g:\mathbb{R}^m \rightarrow \overline{\mathbb{R}}$ are proper, convex and lower semicontinuous functions and $A:\h \rightarrow \mathbb{R}^m$ is a linear continuous operator fulfilling the feasibility condition $A(\dom f) \cap \dom g \neq \emptyset$. The double smoothing technique for solving this class of optimization problems (see \cite{BotHendrich12} for a fully finite-dimensional spaces version of it) assumes to efficiently solve the corresponding Fenchel dual problems and then to recover via an approximately optimal solution of the latter an approximately optimal solution of the primal. This technique, which represents a generalization of the approach developed in \cite{NesterovDoubleSmooth, NesterovDoubleSmoothInfinite} for a special class of convex constrained optimization problems,  makes use of the structure of the Fenchel dual and relies on the regularization of the latter in two steps into an optimization problem having a differentiable strongly convex objective function with Lipschitz continuous gradient. The regularized dual is then solved by a fast gradient method which gives rise to a sequence of dual variables that solve the non-regularized dual problem after $O\left(\frac{1}{\epsilon} \ln\left( \frac{1}{\epsilon}\right)\right)$ iterations, whenever $f$ and $g$ have bounded effective domains. In addition, the norm of the gradient of the regularized dual objective decreases by the same rate of convergence, a fact which is crucial in view of reconstructing an approximately optimal solution to $(P)$ after $O\left(\frac{1}{\epsilon} \ln\left( \frac{1}{\epsilon}\right)\right)$ iterations (see \cite{BotHendrich12}). The first aim of this paper is to show that, whenever $g$ is a strongly convex function, one can obtain the same convergence rate, even without imposing boundedness for its effective domain. Further we show that if, additionally, $f$ is strongly convex or $g$ is everywhere differentiable with a Lipschitz continuous gradient, then the convergence rate becomes $O\left(\frac{1}{\sqrt{\epsilon}} \ln\left( \frac{1}{\epsilon}\right)\right)$, while, if these supplementary assumptions are simultaneous fulfilled, then a convergence rate of $O\left(\ln\left( \frac{1}{\epsilon}\right)\right)$ can be guaranteed.

The structure of the paper is the following. The forthcoming section is dedicated  to some preliminaries on convex analysis and Fenchel duality. In Section \ref{sectionDS} we employ the smoothing technique introduced in \cite{NesterovExcessiveGap05, NesterovSmoothMin05,NesterovSmoothing05} in order to make the objective of the Fenchel dual problem of $(P)$ to be strongly convex and differentiable with Lipschitz continuous gradient. In Section \ref{sectionSolving} we first solve the regularized dual problem  via an efficient fast gradient method. Then we show how do the properties of the functions in the objective of $(P)$ influence the implementation of the double smoothing approach and improve its rate of convergence. We also prove how an approximately optimal primal solution can be recovered from a dual iterate. Finally, in Section \ref{sectionExample}, we consider an application of the presented approach in image deblurring and solve to this end by a linear inverse problem by using two different regularization functionals.

\section{Preliminaries on convex analysis and Fenchel duality}\label{sectionPrelimFormulation}

Throughout this paper $\left\langle \cdot ,\cdot \right\rangle$ and $\left\| \cdot \right\| = \sqrt{\left\langle \cdot, \cdot \right\rangle}$ denote the \textit{inner product} and, respectively, the \textit{norm} of the  Hilbert space $\h$, which is allowed to be infinite dimensional.  The \textit{closure} of a set $C \subseteq \h$ is denoted by $\cl(C)$, while its \textit{indicator function} is the function $\delta_C : \h \rightarrow  \overline{\mathbb{R}} := \mathbb{R} \cup \left\{ \pm \infty \right\}$ defined by $\delta_C(x) = 0$ for $x \in C$ and $\delta_C(x) = +\infty$, otherwise. For a function $f: \h \rightarrow \overline{\mathbb{R}}$ we denote by $\dom f := \left\{ x \in \h : f(x) < +\infty \right\}$ its \textit{effective domain}. We call $f$ \textit{proper} if $\dom f \neq \emptyset$ and $f(x)>-\infty$ for all $x \in \h$. The \textit{conjugate function} of $f$ is $f^*:\h \rightarrow \overline{\mathbb{R}}$, $f^*(p)=\sup{\left\{ \left\langle p,x \right\rangle -f(x) : x\in\h \right\}}$ for all $p \in \h$. The \textit{biconjugate function} of $f$ is $f^{**} : \h \rightarrow \overline \R$, $f^{**}(x) = \sup{\left\{ \left\langle x,p \right\rangle -f^*(p) : p\in\h \right\}}$ and, when $f$ is proper, convex and lower semicontinuous, then, according to the Fenchel-Moreau Theorem, one has $f=f^{**}$. The \textit{(convex) subdifferential} of the function $f$ at $x \in \h$ is the set $\partial f(x) = \{p \in \h : f(y) - f(x) \geq \left\langle p,y-x \right\rangle \ \forall y \in \h\}$, if $f(x) \in \R$, and is taken to be the empty set, otherwise.

Further, we consider the space $\R^m$ endowed with the Euclidean inner product and norm, for which we use the same notations as for the Hilbert space $\h$, since no confusion can arise. By $\mathbbm{1}^m$ we denote the vector in $\R^m$ with all entries equal to $1$. For a subset $C$ of $\R^m$ we denote by $\ri(C)$ its  \textit{relative interior}, i.e. the interior of the set $C$ relative to its affine hull. For a linear continuous operator $A: \h \rightarrow \mathbb{R}^m$ the operator $A^*: \mathbb{R}^m \rightarrow \h$, defined by $\left\langle A^*y,x  \right\rangle = \left\langle y,Ax  \right\rangle$ for all $x \in \h$ and all $y \in \mathbb{R}^m$, is its  so-called \textit{adjoint operator}. By $\id : \R^m \rightarrow \R^m, \id(x) = x,$ for all $x \in \R^m$ we denote the \textit{identity mapping} on $\R^m$.

For a nonempty, convex and closed set $C \subseteq \h$ we consider the \textit{projection operator} $\mathcal{P}_C : \h \rightarrow C$ defined as $x \mapsto \argmin_{z\in C}\left\| x-z \right\|$. Having two functions $f,\,g : \h \rightarrow \overline{\mathbb{R}}$, their \textit{infimal convolution} is defined by $f \Box g : \h \rightarrow \overline{\mathbb{R}}$, $(f \Box g) (x) = \inf_{y \in \h}\left\{ f(y) + g(x-y) \right\} $ for all $x \in \h$. The \textit{Moreau envelope} $^{\gamma}f : \h \rightarrow \overline \R$ of the function $f : \h \rightarrow \overline \R$ of parameter $\gamma > 0$ is defined as the infimal convolution
$$^{\gamma}f(x) := f \Box \left(\frac{1}{2\gamma}\left\| \cdot \right\|^2\right)(x) = \inf_{y \in \h} \left \{f(y) + \frac{1}{2\gamma}\|x-y\|^2 \right\} \ \forall x \in \h.$$
For $\rho > 0$ we say that the function $f : \h \rightarrow \overline \R$ is $\rho$-\textit{strongly convex}, if for all $x,y \in \h$ and all $\lambda \in \left(0,1\right)$ it holds
$$f(\lambda x + (1-\lambda)y) \leq \lambda f(x) + (1-\lambda)f(y) - \frac{\rho}{2}\lambda(1-\lambda)\|x-y\|^2.$$ Notice that this is equivalent to saying that $x \mapsto f(x) - \frac{\rho}{2}\|x\|^2$ is convex.

For the optimization problem $(P)$ we consider the following \textit{standing assumptions}: $f:\h \rightarrow \overline{\mathbb{R}}$ is a proper, convex and lower semicontinuous function with a bounded effective domain, $g:\mathbb{R}^m \rightarrow \overline{\mathbb{R}}$ is proper, $\mu$-strongly convex ($\mu > 0$) and lower semicontinuous function and $A:\h \rightarrow \mathbb{R}^m$ is a linear operator fulfilling $A(\dom f) \cap \dom g \neq \emptyset$.

\begin{remark}
\label{double-smooth-remark-domg-unbounded}
Different to the investigations made in \cite{BotHendrich12} in a fully finite-dimensional setting, we strengthen here the convexity assumptions on $g$ (there $g$ was asked to be only proper, convex and lower semicontinuous), but allow in counterpart $\dom g$ to be unbounded.
\end{remark}

The Fenchel dual problem to $(P)$ (see, for instance, \cite{Bot10,BotGradWanka09}) reads
\begin{equation}\label{opt-problem:dual}
\hspace{-1.8cm}(D) \quad \quad \sup_{p \in \mathbb{R}^m}{ \left\{ -f^*(A^*p)-g^*(-p) \right\} }.
\end{equation}
We denote the optimal objective values of the optimization problems $(P)$ and $(D)$ by $v(P)$ and $v(D)$, respectively.

The conjugate functions of $f$ and $g$ can be written as
\begin{align*}
	f^*(q)
	= \sup_{x \in \dom f}{ \left\{ \left\langle  q,x \right\rangle  -f(x) \right\} }
	= -\inf_{x \in \dom f}{ \left\{ \left\langle -q,x \right\rangle  +f(x) \right\} } \ \forall q \in \h
\end{align*}
and
\begin{align*}
	g^*(p)
	= \sup_{x \in \dom g}{ \left\{ \left\langle  p,x \right\rangle  -g(x) \right\} }
	= -\inf_{x \in \dom g}{ \left\{ \left\langle  -p,x \right\rangle  +g(x) \right\} } \ \forall p \in \R^m,
\end{align*}
respectively. According to \cite[Theorem 11.9]{BauschkeCombettes11} and \cite[Lemma 2.33]{BonnansShapiro00}, the optimization problems arising in the formulation of both $f^*(q)$ for all $q \in \h$ and $g^*(p)$ for all $p \in \R^m$ are solvable, fact which implies that $\dom f^* = \h$ and $\dom g^* = \mathbb{R}^m$, respectively.

By writing the dual problem $(D)$ equivalently as the infimum optimization problem
$$\inf_{p \in \R^m} \{f^*(A^*p) + g^*(-p)\},$$
one can easily see that the Fenchel dual problem of the latter is
$$\sup_{x \in \h} \{-f^{**}(x) - g^{**}(Ax)\}, $$
which, by the Fenchel-Moreau Theorem, is nothing else than
$$\sup_{x \in \h} \{-f(x) - g(Ax)\}.$$
In order to guarantee strong duality for this primal-dual pair it is sufficient to ensure that (see, for instance, \cite[Theorem 2.1]{Bot10}) $0 \in \ri(A^*(\dom g^*) + \dom f^*)$. As $f^*$ has full domain, this regularity condition is automatically fulfilled, which means that $v(D) = v(P)$ and the primal optimization problem $(P)$ has an optimal solution. Due to the fact that $f$ and $g$ are proper and $A(\dom f) \cap \dom g \neq \emptyset$, this further implies $v(D)=v(P) \in \mathbb{R}$. Later we will assume that the dual problem $(D)$ has an optimal solution, too, and that an upper bound of its norm is known.

Denote by $\theta:\mathbb{R}^m \rightarrow \mathbb{R}$, $\theta(p)=f^*(A^*p)+g^*(-p)$, the objective function of $(D)$. Hence, the dual can be equivalently written as
\begin{equation}
\label{opt-problem:dual-in-minimization-form}
\hspace{-1.8cm}(D) \quad \quad -\inf_{p \in \mathbb{R}^m}{ \theta(p) }.
\end{equation}
The assumptions made on $g$ yields that $p \mapsto g^*(-p)$ is differentiable and has a Lipschitz continuous gradient (see Subsection \ref{subsectionFirstSmooth} for details). However, since in general one can not guarantee the smoothness of $p \mapsto f^*(A^*p)$, the dual problem $(D)$ is a nondifferentiable convex optimization problem. Our goal is to solve this problem efficiently and to obtain from here an optimal solution to $(P)$. As in \cite{BotHendrich12}, we are overcoming the non-satisfactory complexity of subgradient-schemes, i.\,e. $O\left(\frac{1}{\epsilon^2}\right)$, by making use of smoothing techniques introduced in \cite{NesterovExcessiveGap05,NesterovSmoothMin05,NesterovSmoothing05}. More precisely, we regularize first the objective function of $f^*(A^*p)$ by a quadratic term in order to obtain a smooth approximation of $p \mapsto f^*(A^*p)$. Then we apply a second regularization to the new dual objective and minimize the regularized problem via an appropriate fast gradient scheme (see \cite{BotHendrich12}). This will allow us to solve both optimization problems $(D)$ and $(P)$ approximately in $O\left(\frac{1}{\epsilon}\ln\left(\frac{1}{\epsilon}\right)\right)$ iterations. More than that, we will show that this rate of convergence can be improved when strengthening the assumptions imposed on $f$ and $g$.

\section{The double smoothing approach}\label{sectionDS}

\subsection{First smoothing}\label{subsectionFirstSmooth}
For a real number $\rho > 0$ the function $p \mapsto f^*(A^*p) = \sup_{x \in \h}{ \left\{ \left\langle  A^*p,x \right\rangle  -f(x) \right\}}$ can be approximated by
\begin{equation}
	\label{opt-problem:f*rho}
	f_{\rho}^*(A^*p) = \sup_{x \in \h}{ \left\{ \left\langle  A^*p,x \right\rangle  -f(x) - \frac{\rho}{2} \left\| x \right\|^2 \right\} }.
\end{equation}
For each $p \in \R^m$ the maximization problem which occurs in the formulation of $f_{\rho}^*(A^*p)$ has a unique solution (see, for instance, \cite[Proposition 11.14]{BauschkeCombettes11}), fact which implies that $f_{\rho}^*(A^*p) \in \R$.

For all $p \in \mathbb{R}^m$ one can express the above regularization of the conjugate by means the Moreau envelope of $f$ as follows
\begin{align*}
	-f_{\rho}^*(A^*p) &= -\sup_{x \in \h}{ \left\{ \left\langle  A^*p,x \right\rangle  -f(x) -\frac{\rho}{2} \left\| x \right\|^2 \right\}} \\
	&= \inf_{x \in \h}{ \left\{ -\left\langle  A^*p,x \right\rangle  +f(x) +\frac{\rho}{2} \left\| x \right\|^2 \right\}} \\
	&= \inf_{x \in \h}{ \left\{ f(x)  +\frac{\rho}{2} \left\| \frac{A^*p}{\rho} -x \right\|^2 \right\}} -\frac{\left\| A^*p\right\|^2}{2\rho}
	= {}^{\frac{1}{\rho}}f\left(\frac{A^*p}{\rho}\right) -\frac{\left\| A^*p\right\|^2}{2\rho}.
\end{align*}
Consequently, one can transfer the differentiability properties of the  Moreau envelope (see \cite[Proposition 12.29]{BauschkeCombettes11}) to $p \mapsto -(f_{\rho}^*\circ A^*)(p)$. For all $p \in \R^m$ we have
\begin{align*}
	- \nabla (f_{\rho}^*\circ A^*)(p) = \frac{A}{\rho}\, \nabla\, {}^{\frac{1}{\rho}}f\left(\frac{A^*p}{\rho}\right) -\frac{AA^*p}{\rho}
	= \frac{A}{\rho}\left( \rho \left( \frac{A^*p}{\rho} - x_{f,p} \right) \right)  -\frac{AA^*p}{\rho}
	= -Ax_{f,p},
\end{align*}
thus
\begin{align*}
	\nabla (f_{\rho}^*\circ A^*)(p) = Ax_{f,p},
\end{align*}
where $x_{f,p} \in \h$ is the \textit{proximal point} of parameter $\tfrac{1}{\rho}$ of $f$ at $\frac{A^*p}{\rho}$, namely the unique element in $\h$ fulfilling (see \cite[Proposition 12.29]{BauschkeCombettes11})
$${}^{\frac{1}{\rho}}f\left(\frac{A^*p}{\rho}\right) =  f(x_{f,p})  +\frac{\rho}{2} \left\| \frac{A^*p}{\rho} -x_{f,p} \right\|^2.$$
By taking into account the nonexpansiveness of the proximal point mapping (see \cite[Proposition 12.27]{BauschkeCombettes11}), for $p,q \in \mathbb{R}^m$ it holds
\begin{align*}
	\left\| \nabla (f_{\rho}^*\circ A^*)(p) - \nabla (f_{\rho}^*\circ A^*)(q) \right\| & = \left\| Ax_{f,p} - Ax_{f,q} \right\| \leq \left\| A \right\| \left\| x_{f,p} - x_{f,q} \right\| \\
	& \leq \left\| A \right\| \left\| \frac{A^*p}{\rho} - \frac{A^*q}{\rho} \right\| \leq \frac{\left\| A \right\|^2}{\rho} \left\| p - q \right\|,
\end{align*}
thus $\frac{\left\| A \right\|^2}{\rho}$ is the Lipschitz constant of $p \mapsto \nabla (f_{\rho}^*\circ A^*)(p)$.

Coming now to the function $p \mapsto g^*(-p) = (g^* \circ -\id)(p)$, let us notice first that, since $g$ is proper, $\mu$-strongly convex and lower semicontinous, $g^*$ is differentiable and $\nabla g^*$ is Lipschitz continuous with Lipschitz constant $\tfrac{1}{\mu}$. Thus $(g^* \circ -\id)$
is Fr\'echet differentiable, too, and its gradient is Lipschitz continuous with Lipschitz constant $\tfrac{1}{\mu}$. By denoting
$$x_{g,p} := \nabla g^*(-p) = -\nabla (g^* \circ -\id)(p),$$
one has that $-p \in \partial g(x_{g,p})$ or, equivalently, $0 \in \partial (\langle p,\cdot \rangle + g)(x_{g,p})$, which means that $x_{g,p}$ is the unique optimal solution (see \cite[Lemma 2.33]{BonnansShapiro00}) of the optimization problem
$$\inf_{x \in \R^m} \{\langle p,x\rangle + g(x)\}.$$

\begin{remark}
\label{remarkRegularizationUnnecessary}
If $f$ is $\rho$-strongly convex, for $\rho > 0$, then there is no need to apply the first regularization for $p \mapsto f^*(A^*p)$, as this function is already Fr\'echet differentiable with a Lipschitz continuous gradient having a Lipschitz constant given by $\frac{\left\| A \right\|^2}{\rho}$. Indeed, the $\rho$-strong convexity of $f$ implies that $f^*$ is Fr\'{e}chet differentiable with Lipschitz continuous gradient having a Lipschitz constant given by $\frac{1}{\rho}$ (see \cite[Theorem 18.15]{BauschkeCombettes11}). Hence, for all $p, q \in \R^m$, we have
\begin{align*}
	\left\| \nabla (f^*\circ A^*)(p) - \nabla (f^*\circ A^*)(q) \right\|
	&= \left\| A \nabla f^*(A^*p) - A \nabla f^*(A^*q) \right\| \\
	%&\leq \left\| A \right\| \left\| \nabla f^*(A^*p) - \nabla f^*(A^*q) \right\|
	&\leq \frac{\left\| A \right\|}{\rho} \left\| A^*p - A^*q \right\|
	\leq \frac{\left\| A \right\|^2}{\rho}  \left\| p - q \right\|.
\end{align*}
By denoting
$$x_{f,p} :=  \nabla f^*(A^*p),$$
one has that $0 \in \partial (f - \langle A^*p,\cdot \rangle)(x_{f,p})$, which means that $x_{f,p}$ is the unique optimal solution (see \cite[Lemma 2.33]{BonnansShapiro00}) of the optimization problem
$$\inf_{x \in \h} \{f(x) - \langle A^*p,x\rangle\}.$$
\end{remark}

By denoting $D_f:=\sup\left\{ \frac{\left\| x \right\|^2}{2} : x \in \dom f  \right\} \in \R$ we can relate $f^* \circ A^*$ and its smooth approximation $f_{\rho}^* \circ A^*$ as follows.
\begin{proposition}
	\label{corollary:f*rho and g*mu inequality}
	For all $p\in\mathbb{R}^m$ it holds
	\begin{align*}
		f_{\rho}^*(A^*p) &\leq f^*(A^*p) \leq f_{\rho}^*(A^*p) + \rho D_f.
\end{align*}
\end{proposition}
\begin{proof} For $p \in \mathbb{R}^m$ one has
\begin{align*}
		f_{\rho}^*(A^*p) &= \left\langle  A^*p,x_{f,p} \right\rangle  -f(x_{f,p}) - \frac{\rho}{2} \left\| x_{f,p} \right\|^2
											\leq \left\langle  A^*p,x_{f,p} \right\rangle  -f(x_{f,p}) \leq  f^*(A^*p) \\
									&\leq \sup_{x \in \dom f}{ \left\{ \left\langle  A^*p,x \right\rangle  -f(x) -\frac{\rho}{2} \left\| x \right\|^2 \right\} } + \sup_{x \in \dom f}{ \left\{\frac{\rho}{2} \left\| x \right\|^2\right\} } \\
									&= f_{\rho}^*(A^*p) + \rho D_f.
\end{align*}
\end{proof}

For $\rho >0$ let $\theta_{\rho}:\mathbb{R}^m \rightarrow \mathbb{R}$ be defined by $\theta_{\rho}(p)=f_{\rho}^*(A^*p) + g^*(-p)$. The function $\theta_{\rho}$ is differentiable with a Lipschitz continuous gradient
$$\nabla \theta_{\rho}(p) = \nabla(f_{\rho}^*\circ A^*)(p) + \nabla (g^* \circ -\id)(p)= Ax_{f,p} - x_{g,p} \ \forall p \in \R^m,$$
having as Lipschitz constant $L(\rho) := \frac{\left\| A \right\|^2}{\rho} + \frac{1}{\mu}$.

In consideration of Proposition \ref{corollary:f*rho and g*mu inequality} we get
\begin{equation}
		\label{inequality:theta-rho-mu and theta relation}
		\theta_{\rho}(p) \leq \theta (p) \leq \theta_{\rho}(p) + \rho D_f \quad \forall p \in \mathbb{R}^m.
\end{equation}

In order to reconstruct an approximately optimal solution to the primal optimization problem $(P)$ it is not sufficient to ensure the convergence of $\theta(\cdot)$ to $-v(D)$, but we also need good convergence properties for the decrease of $\left\| \nabla \theta_{\rho}(\cdot) \right\|$ (cf. \cite{BotHendrich12, NesterovDoubleSmooth}).

\subsection{Second smoothing}\label{subsectionSecondSmoothing}

In the following, a second regularization is applied to $\theta_{\rho}$, as done in \cite{BotHendrich12, NesterovDoubleSmooth, NesterovDoubleSmoothInfinite}, in order to make it strongly convex, fact which will allow us to use a fast gradient scheme with a good convergence rate for the decrease of $\left\| \nabla \theta_{\rho}(\cdot) \right\|$. Therefore, adding the strongly convex function $\frac{\kappa}{2} \left\| \cdot \right\|^2$ to $\theta_{\rho}$, for some positive real number $\kappa$, gives rise to the following regularization of the objective function
\begin{equation*}
\theta_{\rho,\kappa} : \R^m \rightarrow \R, \ \theta_{\rho,\kappa}(p) := \theta_{\rho}(p) + \frac{\kappa}{2} \left\| p \right\|^2
																= f_{\rho}^*(A^*p) + g^*(-p)+ \frac{\kappa}{2} \left\| p \right\|^2,
\end{equation*}
which is obviously $\kappa$-strongly convex. We further deal with the optimization problem
\begin{equation}
		\label{opt-problem:second-regularization}
		\inf_{p \in \mathbb{R}^m}{\theta_{\rho,\kappa}(p)}.
\end{equation}
By taking into account \cite[Lemma 2.33]{BonnansShapiro00}, the optimization problem \eqref{opt-problem:second-regularization} has a unique optimal solution, while the function $\theta_{\rho,\kappa}$ is differentiable and for all $p \in \R^m$ it holds
\begin{align*}
	\nabla \theta_{\rho,\kappa}(p) 	= \nabla \left( \theta_{\rho} (\cdot)+ \frac{\kappa}{2} \left\| \cdot \right\|^2 \right) (p)
																			= Ax_{f,p} - x_{g,p} + \kappa p.
\end{align*}
This gradient is Lipschitz continuous with constant $L(\rho,\kappa):=\frac{\left\| A \right\|^2}{\rho} + \frac{1}{\mu} + \kappa$.
\begin{remark}
\label{remarkSecondRegularizationUnnecessary}
If $\theta_{\rho}$ is $\kappa$-strongly convex, then there is no need to apply the second regularization, as this function is already endowed with the properties of $\theta_{\rho,\kappa}$.
\end{remark}

\section{Solving the doubly regularized dual problem}\label{sectionSolving}

\subsection{A fast gradient method}

In the forthcoming sections we denote by $p_{DS}^*$ the unique optimal solution of the optimization problem \eqref{opt-problem:second-regularization} and by $\theta_{\rho,\kappa}^* := \theta_{\rho,\kappa}(p_{DS}^*)$ its optimal objective value. Further, we denote by $p^* \in \R^m$ an optimal solution to the dual optimization problem $(D)$ and we assume that the upper bound
\begin{equation}
		\label{opt-problem:dual-solution-upper-bound}
		\left\| p^* \right\| \leq R																					
\end{equation}
is available for some nonzero $R \in \R_+$.

Furthermore, we make use of the following fast gradient method (see \cite[Algorithm 2.2.11]{Nesterov04})
\begin{eqnarray}
		\label{doublesmooth:opt-sheme}
		\text{Init.:}& &\text{Set}\ w_0 = p_0:=0 \in \mathbb{R}^m \notag \\
		\text{For } k\geq 0:& &\text{Set}\ p_{k+1} := w_k - \frac{1}{L(\rho,\kappa)}\nabla \theta_{\rho,\kappa}(w_k).\\
			& &\text{Set}\ w_{k+1} := p_{k+1} + \frac{\sqrt{L(\rho,\kappa)} - \sqrt{\kappa}}{\sqrt{L(\rho,\kappa)} + \sqrt{\kappa}} (p_{k+1}-p_k) \notag
\end{eqnarray}
for minimizing the optimization problem \eqref{opt-problem:second-regularization}, which has a strongly convex and differentiable optimization function with a Lipschitz continuous gradient. By taking into account \cite[Theorem 2.2.3]{Nesterov04} we obtain a sequence $(p_k)_{k \geq 0} \subseteq \R^m$ satisfying
\begin{align}
		\theta_{\rho,\kappa}(p_k) - \theta_{\rho,\kappa}^*
		&\leq \left( \theta_{\rho,\kappa}(p_0) - \theta_{\rho,\kappa}^* + \frac{\kappa}{2} \left\| p_0 -p_{DS}^* \right\|^2\right)
					\left( 1-\sqrt{\frac{\kappa}{L(\rho,\kappa)}} \right)^k \notag \\
		\label{opt-scheme:objective-function1}
		&\leq  (\theta_{\rho,\kappa}(p_0) - \theta_{\rho,\kappa}^* + \frac{\kappa}{2} \left\| p_0 -p_{DS}^* \right\|^2)
																																		\,\text{e}^{ -k \sqrt{\frac{\kappa}{L(\rho,\kappa)}} } \\
		\label{opt-scheme:objective-function2}																																
		&\leq  2 (\theta_{\rho,\kappa}(p_0) - \theta_{\rho,\kappa}^*)
																																		\,\text{e}^{ -k \sqrt{\frac{\kappa}{L(\rho,\kappa)}} } \ \forall k \geq 0,
\end{align}
while the last inequality is a consequence of \cite[Theorem 2.1.8]{Nesterov04}. Since $p_{DS}^*$ solves \eqref{opt-problem:second-regularization}, we have $\nabla \theta_{\rho,\kappa}(p_{DS}^*)=0$ and therefore \cite[Theorem 2.1.5]{Nesterov04} yields
\begin{align*}
		\frac{1}{2 L(\rho,\kappa)}\left\| \nabla \theta_{\rho,\kappa}(p_k) \right\|^2
		\leq 	\theta_{\rho,\kappa}(p_k) - \theta_{\rho,\kappa}^*
		\overset{\eqref{opt-scheme:objective-function2}}{\leq} 2 (\theta_{\rho,\kappa}(p_0) - \theta_{\rho,\kappa}^*)
\text{e}^{ -k \sqrt{\frac{\kappa}{L(\rho,\kappa)}}},																																
\end{align*}
which implies
\begin{align}
		\label{opt-scheme:norm-of-gradient}
		\left\| \nabla \theta_{\rho,\kappa}(p_k) \right\|^2
		\leq 	4 L(\rho,\kappa) (\theta_{\rho,\kappa}(p_0) - \theta_{\rho,\kappa}^*)
					\text{e}^{ -k \sqrt{\frac{\kappa}{L(\rho,\kappa)}} }  \ \forall k \geq 0.																																
\end{align}
Due to the $\kappa$-strong convexity of $\theta_{\rho,\kappa}$, \cite[Theorem 2.1.8]{Nesterov04} states
\begin{align}
	\label{opt-scheme:norm-of-arguments-additional}
	\frac{\kappa}{2} \left\| p_k -p_{DS}^* \right\|^2
	\leq \theta_{\rho,\kappa}(p_k) - \theta_{\rho,\kappa}^*
	\overset{\eqref{opt-scheme:objective-function2}}{\leq} 2 (\theta_{\rho,\kappa}(p_0) - \theta_{\rho,\kappa}^*)
\text{e}^{ -k\sqrt{\frac{\kappa}{L(\rho,\kappa)}} }  \ \forall k \geq 0.
\end{align}
Using this inequality it follows that (see also \cite{NesterovDoubleSmooth,NesterovDoubleSmoothInfinite})
\begin{equation}
		\label{opt-scheme:norm-of-arguments}
		\left\| p_k - p_{DS}^* \right\|^2
		\leq 	\min{ \left\{ \left\| p_0 - p_{DS}^* \right\|^2, \frac{4}{\kappa} (\theta_{\rho,\kappa}(p_0) - \theta_{\rho,\kappa}^*)
					\text{e}^{ -k \sqrt{\frac{\kappa}{L(\rho,\kappa)}} }\right\} }	\ \forall k \geq 0.																				
\end{equation}
We first prove that the rates of convergence for the decrease of $\theta(p_k) -\theta(p^*)$ and $\left\| \nabla \theta_{\rho}(p_k) \right\|$ coincide, being equal to $O\left(\frac{1}{\epsilon} \ln \left(\frac{1}{\epsilon}\right)\right)$, and that they can be improved when $f$ and/or $g$ fulfill additional assumptions. We also show how $\epsilon$-optimal solutions to the primal problem $(P)$ can be recovered from the sequence of dual variables $(p_k)_{k \geq 0}$.

\subsection{Convergence of \texorpdfstring{$\theta(p_k)$}{theta(pk)} to \texorpdfstring{$\theta(p^*)$}{theta(p*)}}\label{subsectionConvTheta}

Since the algorithm starts with $p_0 = 0$, we have $\theta_{\rho,\kappa}(0)=f_{\rho}^*(0)+g^*(0)+\frac{\kappa}{2} \left\| 0 \right\|^2 = \theta_{\rho}(0)$, while
\begin{align}
	\label{inequality:theta-rho-mu-kappa in p_{DS}^*}
	\theta_{\rho,\kappa}(p_{DS}^*) = \theta_{\rho}(p_{DS}^*) + \frac{\kappa}{2} \left\| p_{DS}^* \right\|^2.
\end{align}
Making use of these two relations we obtain
\begin{align*}
		\frac{\kappa}{2} \left\| p_{DS}^* \right\|^2
		&\overset{\eqref{opt-scheme:norm-of-arguments-additional}}{\leq} \theta_{\rho,\kappa}(0) - \theta_{\rho,\kappa}(p_{DS}^*)																									 = \theta_{\rho}(0) - \theta_{\rho}(p_{DS}^*) - \frac{\kappa}{2} \left\| p_{DS}^* \right\|^2,
\end{align*}
which further implies that
\begin{align}
		\label{inequality:norm-of-p_DS^*}
		\left\| p_{DS}^* \right\|^2 &\leq \frac{1}{\kappa} \left( \theta_{\rho}(0) - \theta_{\rho}(p_{DS}^*) \right).		
\end{align}
Additionally, in all iterations $k \geq 0$, we have
\begin{align}
		\left\| p_k - p_{DS}^* \right\|^2
		&\overset{\eqref{opt-scheme:norm-of-arguments-additional}}{\leq} \frac{2}{\kappa} \left( \theta_{\rho,\kappa}(p_k) - \theta_{\rho,\kappa}(p_{DS}^*) \right) \notag \\
		&\overset{\eqref{opt-scheme:objective-function1}}{\leq} \frac{2}{\kappa} \left( \theta_{\rho,\kappa}(0) - \theta_{\rho,\kappa}(p_{DS}^*) + \frac{\kappa}{2} \left\| 0 - p_{DS}^* \right\|^2 \right) \text{e}^{ -k \sqrt{\frac{\kappa}{L(\rho,\kappa)}} } \notag \\
		&\overset{\eqref{inequality:theta-rho-mu-kappa in p_{DS}^*}}{=} \frac{2}{\kappa} \left( \theta_{\rho}(0) - \theta_{\rho}(p_{DS}^*) \right) \text{e}^{ -k \sqrt{\frac{\kappa}{L(\rho,\kappa)}} }
		\label{inequality:norm-of-p_k-minus-p_DS^*}
\end{align}
and
\begin{align}
	\label{inequality:theta-rho-mu estimate}
		\theta_{\rho}(p_k) - \theta_{\rho}(p_{DS}^*)
		&\overset{\eqref{opt-scheme:objective-function1}}{\leq} \left( \theta_{\rho,\kappa}(0) - \theta_{\rho,\kappa}(p_{DS}^*) + \frac{\kappa}{2} \left\| 0 - p_{DS}^* \right\|^2 \right)
						\text{e}^{ -k \sqrt{\frac{\kappa}{L(\rho,\kappa)}} } \notag \\
						& \hspace{0.8cm}+ \frac{\kappa}{2} \left( \left\| p_{DS}^* \right\|^2
						- \left\| p_k \right\|^2	\right) \notag \\
		&\overset{\eqref{inequality:theta-rho-mu-kappa in p_{DS}^*}}{=} \left( \theta_{\rho}(0) - \theta_{\rho}(p_{DS}^*)  \right)
						\text{e}^{ -k \sqrt{\frac{\kappa}{L(\rho,\kappa)}} } + \frac{\kappa}{2} \left( \left\| p_{DS}^* \right\|^2
						- \left\| p_k \right\|^2	\right).																					
\end{align}
The estimation
\begin{eqnarray*}
		\left\| p_{DS}^* \right\|^2 - \left\| p_k \right\|^2
		&=& \left(\left\| p_{DS}^* \right\| - \left\| p_k \right\|\right)\left(\left\| p_{DS}^* \right\| + \left\| p_k \right\|\right) \\
		&\leq& \left\| p_{DS}^* - p_k \right\| \left( \left\| p_{DS}^* \right\| + \left\| p_k \right\| \right) \\
		&\leq& \left\| p_{DS}^* - p_k \right\| \left( 2\left\| p_{DS}^* \right\| + \left\| p_k - p_{DS}^* \right\| \right) \\
		&\overset{\eqref{opt-scheme:norm-of-arguments}}{\leq}& 3 \left\| p_{DS}^* - p_k \right\| \left\| p_{DS}^*\right\| \\
		&\overset{\eqref{inequality:norm-of-p_k-minus-p_DS^*}}{\leq}& 3 \left\| p_{DS}^* \right\|
							\sqrt{\frac{2}{\kappa} (\theta_{\rho}(0) - \theta_{\rho}(p_{DS}^*))} \, \text{e}^{ -\frac{k}{2} \sqrt{\frac{\kappa}{L(\rho,\kappa)}} } \\
		&\overset{\eqref{inequality:norm-of-p_DS^*}}{\leq}& 	\frac{3 \sqrt{2}}{\kappa} (\theta_{\rho}(0) - \theta_{\rho}(p_{DS}^*)) \, \text{e}^{ -\frac{k}{2} \sqrt{\frac{\kappa}{L(\rho,\kappa)}} }																					 \end{eqnarray*}
can now be inserted into \eqref{inequality:theta-rho-mu estimate} and this leads to
\begin{align}
		\label{inequality:theta-rho-mu estimate2}
		\theta_{\rho}(p_k) - \theta_{\rho}(p_{DS}^*)
		&\leq \left( \theta_{\rho}(0) - \theta_{\rho}(p_{DS}^*)  \right)
						\left( \text{e}^{ -k \sqrt{\frac{\kappa}{L(\rho,\kappa)}}} + \frac{3}{\sqrt{2}} \, \text{e}^{ -\frac{k}{2} \sqrt{\frac{\kappa}{L(\rho,\kappa)}}} \right) \notag \\
		&\leq \frac{25}{8} \left( \theta_{\rho}(0) - \theta_{\rho}(p_{DS}^*)  \right) \text{e}^{ -\frac{k}{2} \sqrt{\frac{\kappa}{L(\rho,\kappa)}}}  \ \forall k \geq 0.																		 
\end{align}
Further, we have $\theta_{\rho}(0) \overset{\eqref{inequality:theta-rho-mu and theta relation}}{\leq} \theta(0)$, $\theta_{\rho}(p_{DS}^*) \overset{\eqref{inequality:theta-rho-mu and theta relation}}{\geq} \theta(p_{DS}^*) - \rho D_f	\geq \theta(p^*) - \rho D_f$ and, from here,
\begin{align}
		\label{inequality:theta-0-p^*-estimate}
		\theta_{\rho}(0) - \theta_{\rho}(p_{DS}^*)
		\leq \theta(0) - \theta(p^*) + \rho D_f.
\end{align}
Since $\theta_{\rho}(p_{DS}^*) \leq \theta_{\rho}(p_{DS}^*) + \frac{\kappa}{2} \left\| p_{DS}^* \right\|^2 \leq \theta_{\rho}(p^*) + \frac{\kappa}{2} \left\| p^* \right\|^2 $, we obtain that
\begin{align*}
		\theta_{\rho}(p_{DS}^*) \leq \theta_{\rho}(p^*) + \frac{\kappa}{2} \left\| p^* \right\|^2
		\overset{\eqref{inequality:theta-rho-mu and theta relation}}{\leq} \theta(p^*) + \frac{\kappa}{2} \left\| p^* \right\|^2
\end{align*}
and, therefore,
\begin{align}
		\label{inequality:theta-estimate1}
		\theta_{\rho}(p_k) - \theta_{\rho}(p_{DS}^*)
		\overset{\eqref{inequality:theta-rho-mu and theta relation}}{\geq} \theta(p_k) - \rho D_f - \theta(p^*) - \frac{\kappa}{2} \left\| p^* \right\|^2 \ \forall k \geq 0.
\end{align}
In conclusion we obtain for all $k\geq 0$
\begin{eqnarray}
		\label{inequality:theta-estimate2}
		\theta(p_k) - \theta(p^*) &\overset{\eqref{inequality:theta-estimate1}}{\leq}&
															\rho D_f + \frac{\kappa}{2} \left\| p^* \right\|^2 + \theta_{\rho}(p_k) - \theta_{\rho}(p_{DS}^*) \notag \\
		&\overset{\eqref{opt-problem:dual-solution-upper-bound}, \eqref{inequality:theta-rho-mu estimate2}}{\leq}&
				\rho D_f + \frac{\kappa}{2} R^2 + \frac{25}{8} \left( \theta_{\rho}(0) - \theta_{\rho}(p_{DS}^*)  \right) \text{e}^{ -\frac{k}{2} \sqrt{\frac{\kappa}{L(\rho,\kappa)}}} \notag \\
		&\overset{\eqref{inequality:theta-0-p^*-estimate}}{\leq}&
				\rho D_f + \frac{\kappa}{2} R^2 + \frac{25}{8} \left( \theta(0) - \theta(p^*) + \rho D_f  \right) \text{e}^{ -\frac{k}{2} \sqrt{\frac{\kappa}{L(\rho,\kappa)}}}.
\end{eqnarray}
Next we fix $\epsilon > 0$. In order to get $\theta(p_k) - \theta(p^*) \leq \epsilon$ after a certain amount of iterations $k$, we force all three terms in \eqref{inequality:theta-estimate2} to be less than or equal to $\frac{\epsilon}{3}$. To this end we choose first
\begin{align}
	\label{opt-problem-smoothing-parameters}
	\rho := \rho(\epsilon) = \frac{\epsilon}{3 D_f} \ \mbox{and} \	\kappa := \kappa(\epsilon) = \frac{2\epsilon}{3 R^2}.
\end{align}
With these new parameters we can simplify \eqref{inequality:theta-estimate2} to
\begin{equation*}
		\theta(p_k) - \theta(p^*) \leq
		\frac{2 \epsilon}{3} + \frac{25}{8} \left( \theta(0) - \theta(p^*) + \frac{\epsilon}{3}  \right) \text{e}^{ -\frac{k}{2} \sqrt{\frac{\kappa}{L(\rho,\kappa)}}} \ \forall k \geq 0,
\end{equation*}
thus, the second term in the expression on the right-hand side of the above estimate determines the number of iterations needed to obtain $\epsilon$-accuracy for the dual objective function $\theta$. Indeed, we have
\begin{align}
		\frac{\epsilon}{3} &\geq \frac{25}{8} \left( \theta(0) - \theta(p^*) + \frac{\epsilon}{3}  \right) \text{e}^{ -\frac{k}{2} \sqrt{\frac{\kappa}{L(\rho,\kappa)}}} \notag \\
		\Leftrightarrow \text{e}^{ \frac{k}{2} \sqrt{\frac{\kappa}{L(\rho,\kappa)}}} &\geq  \frac{3}{\epsilon} \cdot \frac{25}{8} \left( \theta(0) - \theta(p^*) + \frac{\epsilon}{3}  \right) \notag \\
		\Leftrightarrow \frac{k}{2} \sqrt{\frac{\kappa}{L(\rho,\kappa)}} &\geq \ln{\left( \frac{75\left( \theta(0) - \theta(p^*) + \frac{\epsilon}{3}  \right)}{8 \epsilon}  \right)} \notag \\
		\Leftrightarrow k &\geq 2 \sqrt{\frac{L(\rho,\kappa)}{\kappa}} \ln{\left( \frac{75\left( \theta(0) - \theta(p^*) + \frac{\epsilon}{3}  \right)}{8 \epsilon}  \right)}.  \label{doublesmooth-estimation-theta}
\end{align}
Noticing that
\begin{eqnarray*}
		\frac{L(\rho,\kappa)}{\kappa}
		= \frac{\left\| A \right\|^2}{\rho \kappa} + \frac{1}{\mu \kappa} + 1
		& \overset{\eqref{opt-problem-smoothing-parameters}} {=} & \frac{9 \left\| A \right\|^2 D_f R^2}{2\epsilon^2} + \frac{3 R^2}{2\mu \epsilon} + 1\\
		& = & \frac{1}{\epsilon^2} \left(\frac{9 \left\| A \right\|^2 D_f R^2}{2} + \frac{3 R^2 \epsilon}{2\mu} + \epsilon^2 \right),
\end{eqnarray*}
in order to obtain an approximately optimal solution to $(D)$, we need $k=O\left(\frac{1}{\epsilon}\ln \left(\frac{1}{\epsilon}\right)\right)$ iterations.

\subsection{Convergence of \texorpdfstring{$\left\| \nabla\theta_{\rho}(p_k) \right\|$}{norm nabla-theta-rho} to \texorpdfstring{0}{0}}\label{subsectionConvNablaTheta}

Guaranteeing $\epsilon$-optimality for the objective value of the dual is not sufficient for solving the primal optimization problem with a good convergence rate, as we need at least the same convergence rate for the decrease of $\|\nabla \theta_{\rho}(p_k)\| = \left\| Ax_{f,p_k} - x_{g,p_k} \right\|$ to $0$. Within this section we show that this desiderate is attained (see also \cite{NesterovDoubleSmooth, NesterovDoubleSmoothInfinite}).
Since
\begin{align*}
	\left\| p_k \right\| = \left\| p_k - p_{DS}^* +p_{DS}^* \right\| \leq \left\| p_k - p_{DS}^* \right\| + \left\| p_{DS}^* \right\|
		\overset{\eqref{opt-scheme:norm-of-arguments}}{\leq} 2 \left\| p_{DS}^* \right\| ,
\end{align*}
we conclude that
\begin{align}
	\label{inequality:norm-of-theta-rho-mu}
	\left\| \nabla \theta_{\rho}(p_k) \right\| &\leq \left\|\nabla \theta_{\rho,\kappa}(p_k) \right\| + \left\| \kappa p_k \right\| \notag \\
		&\leq \left\| \nabla\theta_{\rho,\kappa}(p_k) \right\| + 2\kappa \left\| p_{DS}^* \right\|	\ \forall k \geq 0.
\end{align}
We further have
\begin{align*}
	\left\| \nabla\theta_{\rho,\kappa}(p_k) \right\|^2 &\overset{\eqref{opt-scheme:norm-of-gradient}}{\leq}
					4 L(\rho,\kappa) (\theta_{\rho,\kappa}(0) - \theta_{\rho,\kappa}(p_{DS}^*)) \, \text{e}^{ -k \sqrt{\frac{\kappa}{L(\rho,\kappa)}} } \\
					&\overset{\eqref{inequality:theta-rho-mu-kappa in p_{DS}^*}}{\leq}
					4 L(\rho,\kappa) (\theta_{\rho}(0) - \theta_{\rho}(p_{DS}^*))\, \text{e}^{ -k \sqrt{\frac{\kappa}{L(\rho,\kappa)}} } \\
					&\overset{\eqref{inequality:theta-0-p^*-estimate}}{\leq}
					4 L(\rho,\kappa) \left(\theta(0) - \theta(p^*) + \frac{\epsilon}{3}\right)\, \text{e}^{ -k \sqrt{\frac{\kappa}{L(\rho,\kappa)}} },
\end{align*}
which  yields
\begin{align}
	\label{inequality:norm-of-theta-rho-mu-kappa}
	\left\| \nabla\theta_{\rho,\kappa}(p_k) \right\| \leq
					2 \sqrt{L(\rho,\kappa) \left(\theta(0) - \theta(p^*) + \frac{\epsilon}{3}\right)}\, \text{e}^{ -\frac{k}{2} \sqrt{\frac{\kappa}{L(\rho,\kappa)}} }\ \forall k \geq 0.
\end{align}
In order to give an upper bound for the second term in \eqref{inequality:norm-of-theta-rho-mu}, we notice that
\begin{eqnarray*}
	\theta(p^*) + \frac{\kappa}{2} \left\| p^* \right\|^2
	&\overset{\eqref{inequality:theta-rho-mu and theta relation}}{\geq}& \theta_{\rho}(p^*) + \frac{\kappa}{2} \left\| p^* \right\|^2
	\geq \theta_{\rho}(p_{DS}^*) + \frac{\kappa}{2} \left\| p_{DS}^* \right\|^2 \\
	&\overset{\eqref{inequality:theta-rho-mu and theta relation}}{\geq}&
	\theta(p_{DS}^*) -\rho D_f + \frac{\kappa}{2} \left\| p_{DS}^* \right\|^2 \\
	&\geq&
	\theta(p^*) -\rho D_f + \frac{\kappa}{2} \left\| p_{DS}^* \right\|^2,
\end{eqnarray*}
which is equivalent to $\frac{\kappa}{2} \left\| p_{DS}^* \right\|^2 \leq \frac{\kappa}{2} \left\| p^* \right\|^2 + \rho D_f$, i.\,e. $\left\| p_{DS}^* \right\|^2 \leq \left\| p^* \right\|^2 + \frac{2\rho}{\kappa}  D_f$. Hence,
\begin{align}
	\label{inequality:norm of optimal solution p_DS^*}
	\left\| p_{DS}^* \right\|   \leq \sqrt{\left\| p^* \right\|^2 + \frac{2\rho}{\kappa}D_f}
	\overset{\eqref{opt-problem-smoothing-parameters}}{=} \sqrt{\left\| p^* \right\|^2 + \frac{2\epsilon}{3\kappa}}
	\overset{\eqref{opt-problem-smoothing-parameters}}{=} \sqrt{\left\| p^* \right\|^2 + R^2}
	\overset{\eqref{opt-problem:dual-solution-upper-bound}}{\leq} \sqrt{2} R,
\end{align}
which, combined with \eqref{inequality:norm-of-theta-rho-mu} and \eqref{inequality:norm-of-theta-rho-mu-kappa}, provides
\begin{align}
	\label{inequality:norm-final-theta-rho-mu}
	\left\| \nabla\theta_{\rho}(p_k) \right\| &\leq
	2 \sqrt{L(\rho,\kappa) \left(\theta(0) - \theta(p^*) + \frac{\epsilon}{3}\right)} \,\text{e}^{ -\frac{k}{2} \sqrt{\frac{\kappa}{L(\rho,\kappa)}} }
	+ 2\sqrt{2} \kappa R \notag \\
	&= 2 \sqrt{L(\rho,\kappa) \left(\theta(0) - \theta(p^*) + \frac{\epsilon}{3}\right)} \,\text{e}^{ -\frac{k}{2} \sqrt{\frac{\kappa}{L(\rho,\kappa)}} }	+ \frac{4\sqrt{2}\epsilon}{3R} \ \forall k \geq 0.
\end{align}
For  $\epsilon > 0$ fixed, the first term in \eqref{inequality:norm-final-theta-rho-mu} decreases by the iteration counter $k$, and, by taking into account \eqref{opt-problem-smoothing-parameters}, we can ensure
\begin{equation}
	\label{opt-scheme:convergence-for-dual-problem}
	\theta(p_k) -\theta(p^*) \leq \epsilon \ \mbox{and} \ \left\| \nabla \theta_{\rho}(p_k) \right\| \leq \frac{2\epsilon}{R}
\end{equation}
in $k=O\left(\frac{1}{\epsilon} \ln\left(\frac{1}{\epsilon}\right)\right)$ iterations.

\subsection{Improved convergence rates}
In this  subsection we investigate how additionally assumptions on the functions $f$ and/or $g$ influence the implementation of the double smoothing approach and its rate of convergence.

\subsubsection{The case \texorpdfstring{$f$}{f} is strongly convex}\label{subsectionSmoothSmooth}
Assuming additionally to the standing assumptions that the function $f:\h \rightarrow \overline{\mathbb{R}}$ is $\rho$-strongly convex,  for $\rho > 0$, the first smoothing, as done in Subsection \ref{subsectionFirstSmooth}, can be omitted and the fast gradient method \eqref{doublesmooth:opt-sheme} can be applied to the function $\theta_{\kappa}:\R^m \rightarrow \R$, $\theta_{\kappa}:= f^*(A^*p) + g^*(-p) + \frac{\kappa}{2}\left\| p \right\|^2$, with $\kappa>0$, which is $\kappa$-strongly convex and differentiable with Lipschitz continuous gradient. In the light of Remark \ref{remarkRegularizationUnnecessary} the Lipschitz constant of $\nabla \theta_{\kappa}$ is $L(\kappa) := \frac{\left\| A \right\|^2}{\rho} + \frac{1}{\mu} + \kappa$.

Similar to the calculations made in Section \ref{subsectionConvTheta} we obtain for all $k \geq 0$
\begin{align*}
		\theta(p_k) - \theta(p^*) \leq \frac{\kappa}{2} R^2 + \frac{25}{8} \left( \theta(0) - \theta(p^*) \right) \text{e}^{ -\frac{k}{2} \sqrt{\frac{\kappa}{L(\kappa)}}}.
\end{align*}
Hence, when $\varepsilon > 0$, in order to guarantee $\epsilon$-accuracy for the dual objective function we can force both terms in the above estimate to be less than or equal to $\frac{\epsilon}{2}$. Thus, by taking
\begin{align*}
	\kappa := \kappa(\epsilon) = \frac{\epsilon}{R^2},
\end{align*}
this time we will need to this end, in contrast to \eqref{doublesmooth-estimation-theta},
$$ k \geq 2 \sqrt{\frac{L(\kappa)}{\kappa}} \ln{\left( \frac{25\left( \theta(0) - \theta(p^*) \right)}{4 \epsilon}  \right)}, $$
i.\,e. $k=O\left(\frac{1}{\sqrt{\epsilon}} \ln\left(\frac{1}{\epsilon}\right)\right)$ iterations.

In analogy to the considerations made in Section \ref{subsectionConvNablaTheta} we obtain for all $k \geq 0$
\begin{align*}
	\left\| \nabla \theta(p_k) \right\| &\leq 2 \sqrt{L(\kappa)(\theta(0)-\theta(p^*))}\,\text{e}^{-\frac{k}{2}\sqrt{\frac{\kappa}{L(\kappa)}}} +2\kappa R \\
	&= 2 \sqrt{L(\kappa)(\theta(0)-\theta(p^*))}\,\text{e}^{-\frac{k}{2}\sqrt{\frac{\kappa}{L(\kappa)}}} + \frac{2\epsilon}{R}.
\end{align*}
Therefore, in order to guarantee $\left\| \nabla \theta(p_k) \right\| \leq \frac{3\epsilon}{R}$, we need $k= O\left(\frac{1}{\sqrt{\epsilon}} \ln\left(\frac{1}{\epsilon}\right)\right)$ iterations, which coincide with the convergence rate for the dual objective values.

\subsubsection{The case \texorpdfstring{$g$}{g} is everywhere differentiable with Lipschitz continuous gradient}\label{subsectionSmoothNonsmooth}
Assuming additionally to the standing assumptions that the function $g:\R^m \rightarrow \R$ has full domain and it is differentiable with $\frac{1}{\kappa}$-Lipschitz continuous gradient, for $\kappa > 0$, the second smoothing, as done in Subsection \ref{subsectionSecondSmoothing} can be omitted. The fast gradient method \eqref{doublesmooth:opt-sheme} can be applied to the function $\theta_{\rho}:\R^m \rightarrow \R$, $\theta_{\rho}:= f_{\rho}^*(A^*p) + g^*(-p)$, which is $\kappa$-strongly convex due to \cite[Theorem 18.15]{BauschkeCombettes11} and differentiable with Lipschitz continuous gradient. The Lipschitz constant of $\nabla \theta_{\rho}$ is $L(\rho) := \frac{\left\| A \right\|^2}{\rho} + \frac{1}{\mu}$.

The algorithm \eqref{doublesmooth:opt-sheme} applied to $\theta_{\rho}$ states
\begin{align*}
		\theta_{\rho}(p_k) - \theta_{\rho}(p_{DS}^*)
		&\leq \left( \theta_{\rho}(0) - \theta_{\rho}(p_{DS}^*) + \frac{\kappa}{2} \left\| 0 - p_{DS}^*\right\|^2 \right) \text{e}^{ -k \sqrt{\frac{\kappa}{L(\rho)}}} \\
		&\leq 2\left( \theta_{\rho}(0) - \theta_{\rho}(p_{DS}^*) \right) \text{e}^{ -k \sqrt{\frac{\kappa}{L(\rho)}}} \ \forall k \geq 0.
\end{align*}
Since $\theta_{\rho}(0) \overset{\eqref{inequality:theta-rho-mu and theta relation}}{\leq} \theta(0)$ and $\theta_{\rho}(p_{DS}^*) \overset{\eqref{inequality:theta-rho-mu and theta relation}}{\geq} \theta(p_{DS}^*) - \rho D_f \geq \theta(p^*) - \rho D_f$, we obtain
\begin{align}
	\label{inequality-smooth-nonsmooth-section}
	\theta_{\rho}(0) - \theta_{\rho}(p_{DS}^*) \leq \theta(0) - \theta(p^*) + \rho D_f.
\end{align}
On the other hand, since $\theta_{\rho}(p_k) - \theta_{\rho}(p_{DS}^*) \overset{\eqref{inequality:theta-rho-mu and theta relation}}{\geq} \theta(p_k) - \rho D_f -\theta(p^*)$, it follows
\begin{align*}
		\theta(p_k) - \theta(p^*)
		&\leq \rho D_f + \theta_{\rho}(p_k) - \theta_{\rho}(p_{DS}^*) \\
		&\leq \rho D_f + 2\left( \theta(0) - \theta(p^*) +\rho D_f \right) \text{e}^{ -k \sqrt{\frac{\kappa}{L(\rho)}}} \ \forall k \geq 0.
\end{align*}
Hence, when $\varepsilon > 0$, in order to guarantee $\epsilon$-optimality for the dual objective, we force both terms in the above estimate less than or equal to $\frac{\epsilon}{2}$. By taking
\begin{align}
	\label{smoothing-parameter-smooth-nonsmooth-section}
	\rho := \rho(\epsilon) = \frac{\epsilon}{2 D_f},
\end{align}
in contrast to \eqref{doublesmooth-estimation-theta}, we need
$$ k \geq \sqrt{\frac{L(\rho)}{\kappa}} \ln{\left( \frac{4\left( \theta(0) - \theta(p^*) + \frac{\epsilon}{2}\right)}{\epsilon}  \right)}, $$
i.\,e. $k=O\left(\frac{1}{\sqrt{\epsilon}} \ln\left(\frac{1}{\epsilon}\right)\right)$ iterations.

We obtain as well
\begin{align*}
	\left\| \nabla \theta_{\rho}(p_k) \right\| &\overset{\eqref{opt-scheme:norm-of-gradient}}{\leq} 2 \sqrt{L(\rho)(\theta_{\rho}(0)-\theta_{\rho}(p_{DS}^*))}\,\text{e}^{-\frac{k}{2}\sqrt{\frac{\kappa}{L(\rho)}}} \\
	&\overset{\eqref{inequality-smooth-nonsmooth-section}}{\leq} 2 \sqrt{L(\rho)(\theta(0)-\theta(p^*) + \rho D_f)}\,\text{e}^{-\frac{k}{2}\sqrt{\frac{\kappa}{L(\rho)}}} \\
	&\overset{\eqref{smoothing-parameter-smooth-nonsmooth-section}}{=} 2 \sqrt{L(\rho)(\theta(0)-\theta(p^*) + \frac{\epsilon}{2})}\,\text{e}^{-\frac{k}{2}\sqrt{\frac{\kappa}{L(\rho)}}} \ \forall k \geq 0.
\end{align*}
Therefore, in order to guarantee $\left\| \nabla \theta(p_k) \right\| \leq \frac{3\epsilon}{R}$, we need $k= O\left(\frac{1}{\sqrt{\epsilon}} \ln\left(\frac{1}{\epsilon}\right)\right)$ iterations, which is the same convergence rate as for the dual objective values.

\subsubsection{The case \texorpdfstring{$f$}{f} is strongly convex and \texorpdfstring{$g$}{g} is everywhere differentiable with Lipschitz continuous gradient}\label{subsectionSmoothSmoothStrongly}
Assuming additionally to the standing assumptions that the function $f:\h \rightarrow \overline{\mathbb{R}}$ is $\rho$-strongly convex, for $\rho > 0$, and the function $g:\R^m \rightarrow \R$ has full domain and it is  differentiable with $\frac{1}{\kappa}$-Lipschitz continuous gradient, for $\kappa > 0$, both the first and second smoothing can be omitted. The fast gradient method \eqref{doublesmooth:opt-sheme} can be applied to the function $\theta :\R^m \rightarrow \R$, $\theta:= f^*(A^*p) + g^*(-p)$, which is $\kappa$-strongly convex and differentiable with Lipschitz continuous gradient. The Lipschitz constant of $\nabla \theta$ is $L:= \frac{\left\| A \right\|^2}{\rho} + \frac{1}{\mu}$.

The fast gradient scheme \eqref{doublesmooth:opt-sheme} applied to $\theta$ yields for all $k \geq 0$
\begin{align*}
		\theta(p_k) - \theta(p^*)
		 \overset{\eqref{opt-scheme:objective-function1}}{\leq}  (\theta(0) - \theta(p^*) + \frac{\kappa}{2} \left\|0-p^* \right\|^2)
																																		\,\text{e}^{ -k \sqrt{\frac{\kappa}{L}} }
		 \overset{\eqref{opt-scheme:objective-function2}}{\leq}  2 (\theta(0) - \theta(p^*)) \,\text{e}^{ -k \sqrt{\frac{\kappa}{L}} }
\end{align*}
and, from here, when $\varepsilon > 0$,
\begin{align*}
		2 (\theta(0) - \theta(p^*)) \,\text{e}^{ -k \sqrt{\frac{\kappa}{L}} } \leq \epsilon
		%\Leftrightarrow \text{e}^{ k \sqrt{\frac{\kappa}{L}} } &\geq \frac{2}{\epsilon} (\theta(0) - \theta(p^*)) \\
		\Leftrightarrow k \geq \sqrt{\frac{L}{\kappa}} \ln \left( \frac{2(\theta(0) - \theta(p^*))}{\epsilon} \right)
\end{align*}
On the other hand, formula \eqref{opt-scheme:norm-of-gradient} states $\left\| \nabla \theta(p_k) \right\|	\leq 	2 \sqrt{L (\theta(0) - \theta(p^*))}
\,\text{e}^{ -\frac{k}{2} \sqrt{\frac{\kappa}{L}} }$ for all $k\geq 0$, thus
\begin{align*}
		2 \sqrt{L(\theta(0) - \theta(p^*))} \,\text{e}^{ - \frac{k}{2} \sqrt{\frac{\kappa}{L}} }  \leq \epsilon
		%\Leftrightarrow \text{e}^{ \frac{k}{2} \sqrt{\frac{\kappa}{L}} } &\geq \frac{2}{\epsilon} \sqrt{L(\theta(0) - \theta(p^*))}  \\
		\Leftrightarrow k \geq 2\sqrt{\frac{L}{\kappa}} \ln \left( \frac{2\sqrt{L(\theta(0) - \theta(p^*))}}{\epsilon} \right).
\end{align*}	
In conclusion, in order to guarantee $\varepsilon$-accuracy for the dual objective values and for the decrease of $\|\nabla \theta(\cdot)\|$ to $0$, we need $O\left(\ln\left(\frac{1}{\epsilon}\right)\right)$ iterations.

\subsection{Constructing an approximate primal solution}\label{subsectionApproximatePrimal}
In the remaining of this section we work in the setting of our initial standing assumptions and show, first of all, how to recover approximately optimal solutions for the primal $(P)$  from the sequence of approximately dual solutions $(p_k)_{k \geq 0}$. This will be followed by a convergence analysis for the approximate primal optimal solutions. One can easily notice that the investigations made here remain valuable when working in the special settings of the previous section, too.

Since our main focus is to solve the primal optimization problem $(P)$, we prove as follows that the sequences $(x_{f,p_k})_{k\geq0} \subseteq \dom f$ and $(x_{g,p_k})_{k\geq0} \subseteq \dom g$  constructed in Subsection \ref{subsectionFirstSmooth} contain all the information one needs to recover approximately optimal solutions to $(P)$.

Since $\theta_{\rho}(p_k) -\theta(p^*) \overset{\eqref{inequality:theta-rho-mu and theta relation}}{\leq} \theta(p_k) -\theta(p^*) \leq \epsilon$ and
\begin{align*}
		\theta_{\rho}(p_k) - \theta(p^*) \overset{\eqref{inequality:theta-rho-mu and theta relation}}{\geq} \theta(p_k) - \rho D_f -\theta(p^*)
		\overset{\eqref{opt-problem-smoothing-parameters}}{=} \underbrace{\theta(p_k) -\theta(p^*)}_{\geq 0} -\frac{\epsilon}{3} \geq -\frac{\epsilon}{3},
\end{align*}
it holds $\left| \theta_{\rho}(p_k) -\theta(p^*) \right| \leq \epsilon$ for all $k \geq 0$. Further, for $p_k \in \mathbb{R}^m$ we have
\begin{eqnarray*}
		\theta_{\rho}(p_k) & = &	f_{\rho}^*(A^*p_k) + g^*(-p_k)\\
                           & = & \left\langle  p_k,Ax_{f,p_k} \right\rangle  -f(x_{f,p_k}) - \frac{\rho}{2} \left\| x_{f,p_k} \right\|^2
				  -\left\langle  p_k,x_{g,p_k} \right\rangle  -g(x_{g,p_k})
\end{eqnarray*}
and from here (notice that $-v(D)=\theta(p^*)$)
\begin{align*}
		f(x_{f,p_k}) + g(x_{g,p_k}) -v(D) = \left\langle p_k,\nabla \theta_{\rho}(p_k) \right\rangle + (\theta(p^*) -\theta_{\rho}(p_k) ) - \frac{\rho}{2} \left\| x_{f,p_k} \right\|^2 \ \forall k \geq 0.
\end{align*}
It follows
\begin{eqnarray*}
		\left| f(x_{f,p_k}) + g(x_{g,p_k}) -v(D) \right|
		&\leq& \left\| p_k \right\| \left\| \nabla \theta_{\rho}(p_k) \right\| + \left| \theta(p^*) -\theta_{\rho}(p_k) \right| + \frac{\rho}{2} \left\| x_{f,p_k} \right\|^2 \\
		&\leq& \left\| p_k \right\| \left\| \nabla \theta_{\rho}(p_k) \right\| + \epsilon + \rho D_f  \\
		&\overset{\eqref{opt-problem-smoothing-parameters}}{\leq}& \left\| p_k \right\| \left\| \nabla \theta_{\rho}(p_k) \right\| + 2 \epsilon
		\overset{\eqref{opt-scheme:convergence-for-dual-problem}}{\leq} \frac{2\epsilon}{R} \left\| p_k \right\| +2 \epsilon \ \forall k \geq 0.
\end{eqnarray*}
Further, $\left\| p_k \right\|$ can be estimated above using
\begin{align*}
	\left\| p_k \right\| = \left\| p_k + p_{DS}^* - p_{DS}^* \right\| \leq \left\| p_k -p_{DS}^* \right\| + \left\| p_{DS}^* \right\|
	\overset{\eqref{opt-scheme:norm-of-arguments}}{\leq} 2 \left\| p_{DS}^* \right\| \overset{\eqref{inequality:norm of optimal solution p_DS^*}}{\leq} 2\sqrt{2}R,
\end{align*}
therefore, we obtain
\begin{align}
		\label{inequality:primal-solution-estimate}
		\left| f(x_{f,p_k}) + g(x_{g,p_k}) -v(D) \right| \leq 4 \sqrt{2} \epsilon + 2 \epsilon = 2(2\sqrt{2}+1) \epsilon \ \forall k \geq 0.
\end{align}
By taking into account weak duality, i.\,e. $v(D)\leq v(P)$, we conclude that $x_{f,p_k} \in \dom f$ and $x_{g,p_k} \in \dom g$ can be seen as approximately optimal solutions to $(P)$ when $k$ is high enough to satisfy \eqref{opt-scheme:convergence-for-dual-problem}.

\subsection{Existence of an optimal solution}\label{subsectionExistence}

This section is devoted to the convergence analysis of our primal sequences when $\varepsilon$ converges to zero. To this end let $(\epsilon_n)_{n\geq0} \subseteq \mathbb{R}_+$ be a decreasing sequence of positive scalars with $\lim_{n\rightarrow \infty}{\epsilon_n}=0$. For each $n\geq0$, the double smoothing algorithm  \eqref{doublesmooth:opt-sheme} with smoothing parameters $\rho_{\epsilon_n}$ and $\kappa_{\epsilon_n}$ given by \eqref{opt-problem-smoothing-parameters} requires at least $k=k(\epsilon_n)$ iterations to fulfill \eqref{opt-scheme:convergence-for-dual-problem}. For $n \geq 0$ we denote
\begin{equation*}
	\bx_n := x_{f,p_{k(\epsilon_n)}}  \in \dom f \ \mbox{and} \ \by_n := x_{g,p_{k(\epsilon_n)}}  \in \dom g.
\end{equation*}
Due to the boundedness of $\dom f$, its closure $\cl(\dom f)$ is weakly compact (see \cite[Theorem 3.3]{BauschkeCombettes11})  and there exists a subsequence $(\bx_{n_l})_{l \geq 0}$ and $\bx \in \h$ such that
$\bx_{n_l}$ weakly converges to $\bx \in \cl(\dom f)$ when $l \rightarrow +\infty$. Since $A:\h \rightarrow \R^m$ is linear and continuous, the sequence $A\bx_{n_l}$ will converge to $A\bx$ when $l \rightarrow +\infty$. In view of relation \eqref{opt-scheme:convergence-for-dual-problem} we get
\begin{align}
		\label{doublesmooth-feasible-estimate}
		0 \leq  \left\| A\bx_{n_l} - \by_{n_l} \right\| &\leq \frac{2\epsilon_{n_l}}{R} \ \forall l \geq 0.
\end{align}
This means that the sequence $(\by_{n_l})_{l \geq 0} \subseteq \dom g$ is obviously bounded, hence there exists a subsequence of it (still denoted by $(\by_{n_l})_{l \geq 0}$) and an element $\bar y \in \cl(\dom g)$ such that
$\by_{n_l} \rightarrow \by$ when $l \rightarrow +\infty$. Taking $l \rightarrow +\infty$ in \eqref{doublesmooth-feasible-estimate} it follows $A\bx=\by$.
Furthermore, due to \eqref{inequality:primal-solution-estimate}, we have
\begin{align*}
	f(\bx_{n_l}) + g(\by_{n_l}) \leq v(D) + 2(3\sqrt{2}+1)\epsilon_{n_l} \quad \forall l \geq 0
\end{align*}
and, by using the lower semicontinuity of $f$ and $g$ and \cite[Theorem 9.1]{BauschkeCombettes11}, we obtain
\begin{eqnarray*}
	f(\bx) + g(A\bx) & \leq & \liminf_{l \rightarrow \infty}{\left\{f(\bx_{n_l}) + g(\by_{n_l})\right\}}\\
	& \leq & \lim_{l \rightarrow \infty}{\left\{v(D) + 2(3\sqrt{2}+1)\epsilon_{n_l}\right\}} = v(D) \leq v(P).
\end{eqnarray*}
Since $v(P) \in \R$, we have $\bx \in \dom f$ and $A\bx \in \dom g$, which yields that $\bx$ is an optimal solution to $(P)$.

\section{Two examples in image processing}\label{sectionExample}
In this section we are solving a linear inverse problem which arises in the field of signal and image processing via the double smoothing algorithm developed in this paper. For a given matrix $A \in \mathbb{R}^{n \times n}$ describing a \textit{blur operator} and a given vector $b \in \R^n$ representing the \textit{blurred and noisy image} the task is to estimate the \textit{unknown original image} $x^*\in\R^n$ fulfilling
$$Ax=b.$$
To this end we make use of two regularization functionals with different properties.

\subsection{An \texorpdfstring{$l_1$}{l1} regularization problem}\label{subsectionExample-l1}

We start by solving the $l_1$ regularized convex optimization problem
\begin{align*}
	\hspace{-1.8cm}(P) \quad \quad \inf_{x \in S}{\left\{ \left\| Ax-b \right\|^2 + \lambda \left\| x \right\|_1\right\}},
\end{align*}
where $S\subseteq \R^n$ is an $n$-dimensional cube representing the range of the pixels and $\lambda > 0$  the regularization parameter. The problem to be solved can be equivalently written as
\begin{align*}
	\hspace{-1.8cm}(P) \quad \quad \inf_{x \in \mathbb{R}^n}{\left\{ f(x) + g(Ax)\right\}},
\end{align*}
for $f:\R^n \rightarrow \overline{\R}$, $f(x)=\lambda\left\| x \right\|_1 + \delta_{S}(x)$ and $g:\R^n \rightarrow \R$, $g(y)=\left\|y-b \right\|^2$. Thus $f$ is proper, convex and lower semicontinuous with bounded domain and $g$ is a $2$-strongly convex function with full domain, differentiable everywhere and with Lipschitz continuous gradient having as Lipschitz constant $2$. This means that we are in the setting of Subsection \ref{subsectionSmoothNonsmooth}.

By making use of gradient methods, both the iterative shrinkage-tresholding algorithm (ISTA) (see \cite{Daubechiesetal04}) and its accelerated variant FISTA (see \cite{BeckTeboulle09, BeckTeboulle10}) solve the optimization problem $(P)$ in $O\left(\frac{1}{\epsilon}\right)$ and $O\left(\frac{1}{\sqrt{\epsilon}}\right)$ iterations, respectively, whereas the convergence rate of our method is $O\left(\frac{1}{\sqrt{\epsilon}} \ln\left(\frac{1}{\epsilon}\right)\right)$.

Since each pixel furnishes a greyscale value which is between $0$ and $255$, a natural choice for the convex set $S$ would be the $n$-dimensional cube $\left[0,255\right]^n \subseteq \mathbb{R}^n$. In order to reduce the Lipschitz constant  which appears in the developed approach, we scale the pictures to which refer within this subsection such that each of their  pixels ranges in the interval $\left[0,\frac{1}{10}\right]$. We concretely look at the $256 \times 256$ \textit{cameraman test image}, which is part of the image processing toolbox in Matlab. The dimension of the vectorized and scaled cameraman test image is $n=256^2=65536$. By making use of the Matlab functions {\ttfamily imfilter} and {\ttfamily fspecial}, this image is blurred as follows:
\begin{lstlisting}[numbers=left,numberstyle=\tiny,frame=tlrb,showstringspaces=false]
H=fspecial('gaussian',9,4);   % gaussian blur of size 9 times 9
                              % and standard deviation 4
B=imfilter(X,H,'conv','symmetric');  % B=observed blurred image
                                     % X=original image		
\end{lstlisting}
In row $1$ the function {\ttfamily fspecial} returns a rotationally symmetric Gaussian lowpass filter of size $9 \times 9$ with standard deviation $4$. The entries of $H$ are nonnegative and their sum adds up to $1$. In row $3$ the function {\ttfamily imfilter} convolves the filter $H$  with the image $X\in \mathbb{R}^{256 \times 256}$ and outputs the blurred image $B\in \mathbb{R}^{256 \times 256}$ . The boundary option "symmetric" corresponds to reflexive boundary conditions.

Thanks to the rotationally symmetric filter $H$, the linear operator $A\in\mathbb{R}^{n \times n}$ given by the Matlab function {\ttfamily imfilter} is symmetric, too. By making use of the real spectral decomposition of $A$, it shows that $\left\| A \right\|^2=1$. After adding a zero-mean white Gaussian noise with standard deviation $10^{-4}$, we obtain the blurred and noisy image $b \in \mathbb{R}^n$ which is shown in Figure \ref{fig:cameraman}.
\begin{figure}[ht]
	\centering
	\includegraphics*[viewport= 58 318 560 553, width=0.8\textwidth]{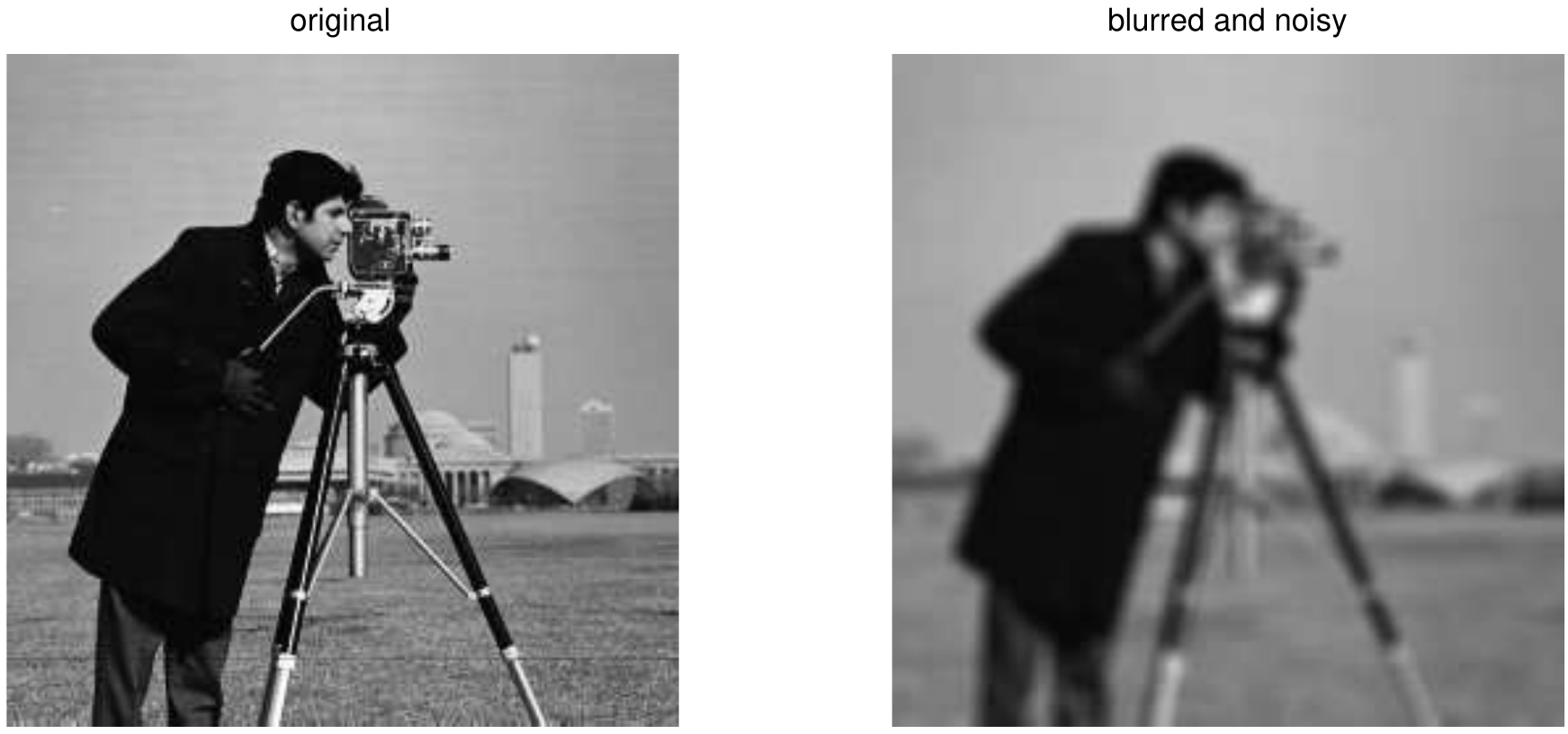}
	\caption{The $256 \times 256$ cameraman test image}
	\label{fig:cameraman}
\end{figure}

The dual optimization problem in minimization form is
\begin{align*}
	\hspace{-1.8cm}(D) \quad \quad -\inf_{p \in \mathbb{R}^n}{\left\{f^*(A^*p)+g^*(-p)\right\}}
\end{align*}
and, due to the fact that $g$ has full domain, strong duality for $(P)$ and $(D)$ holds, i.\,e. $v(P)=v(D)$ and $(D)$ has an optimal solution (see, for instance, \cite{Bot10,BotGradWanka09}). By taking into consideration \eqref{smoothing-parameter-smooth-nonsmooth-section}, the smoothing parameter is taken as
\begin{align}
	\label{example-smoothing-parameters}
	\rho := \frac{\epsilon}{2D_f}
\end{align}
for $D_f = \sup{\left\{ \frac{\left\|x\right\|^2}{2} : x \in \left[0,\frac{1}{10}\right]^n \right\}} = 327.68$, while the accuracy is chosen to be $\epsilon=0.3$ and the regularization parameter is set to $\lambda=2$e-$6$.

We show next that the sequences of approximate primal solutions $(x_{f,p_k})_{k \geq 0}$ and $(x_{g,p_k})_{k \geq 0}$ can be easily calculated. Indeed, for $k \geq 0$ we have
\begin{eqnarray*}
	 x_{f,p_k} & = & \argmin_{x\in \left[0,\frac{1}{10}\right]^n}{\left\{ \lambda \left\|x\right\|_1 + \frac{\rho}{2}\left\|\frac{A^*p_k}{\rho} - x\right\|^2 \right\}}\\
& = & \argmin_{x\in \left[0,\frac{1}{10}\right]^n}{\left\{ \sum_{i=1}^{n}{\left[ \lambda \left|x_i\right| + \frac{\rho}{2} \left(\frac{(A^*p_k)_i}{\rho} - x_i \right)^2\right]}  \right\}}
\end{eqnarray*}
and, in order to determine it, we need to solve the one-dimensional convex optimization problem
\begin{align*}
	 \inf_{x_i\in \left[0,\frac{1}{10}\right]}{\left\{ \lambda x_i + \frac{\rho}{2} \left(\frac{(A^*p_k)_i}{\rho} - x_i \right)^2 \right\}},
\end{align*}
for $i=1,\ldots,n$, which has as unique optimal solution $\mathcal{P}_{\left[0,\frac{1}{10}\right]}\left( \frac{1}{\rho} \left( (A^*p_k)_i - \lambda \right) \right)$.
Thus,
$$x_{f,p_k} = \mathcal{P}_{\left[0,\frac{1}{10}\right]^n}\left( \frac{1}{\rho} \left( A^*p_k - \lambda \mathbbm{1}^n \right) \right).$$
On the other hand, for all $k \geq 0$ we have
\begin{align*}
x_{g,p_k} = \argmin_{x\in\R^n}{\left\{ \left\langle p_k,x \right\rangle + g(x) \right\}}
= \argmin_{x\in\R^n}{\left\{ \left\langle p_k,x \right\rangle + \left\|x - b\right\|^2 \right\}}
= b-\frac{1}{2}p_k.
\end{align*}

\begin{figure}[ht]	
	\centering
	\includegraphics*[viewport= 143 249 470 606, width=0.32\textwidth]{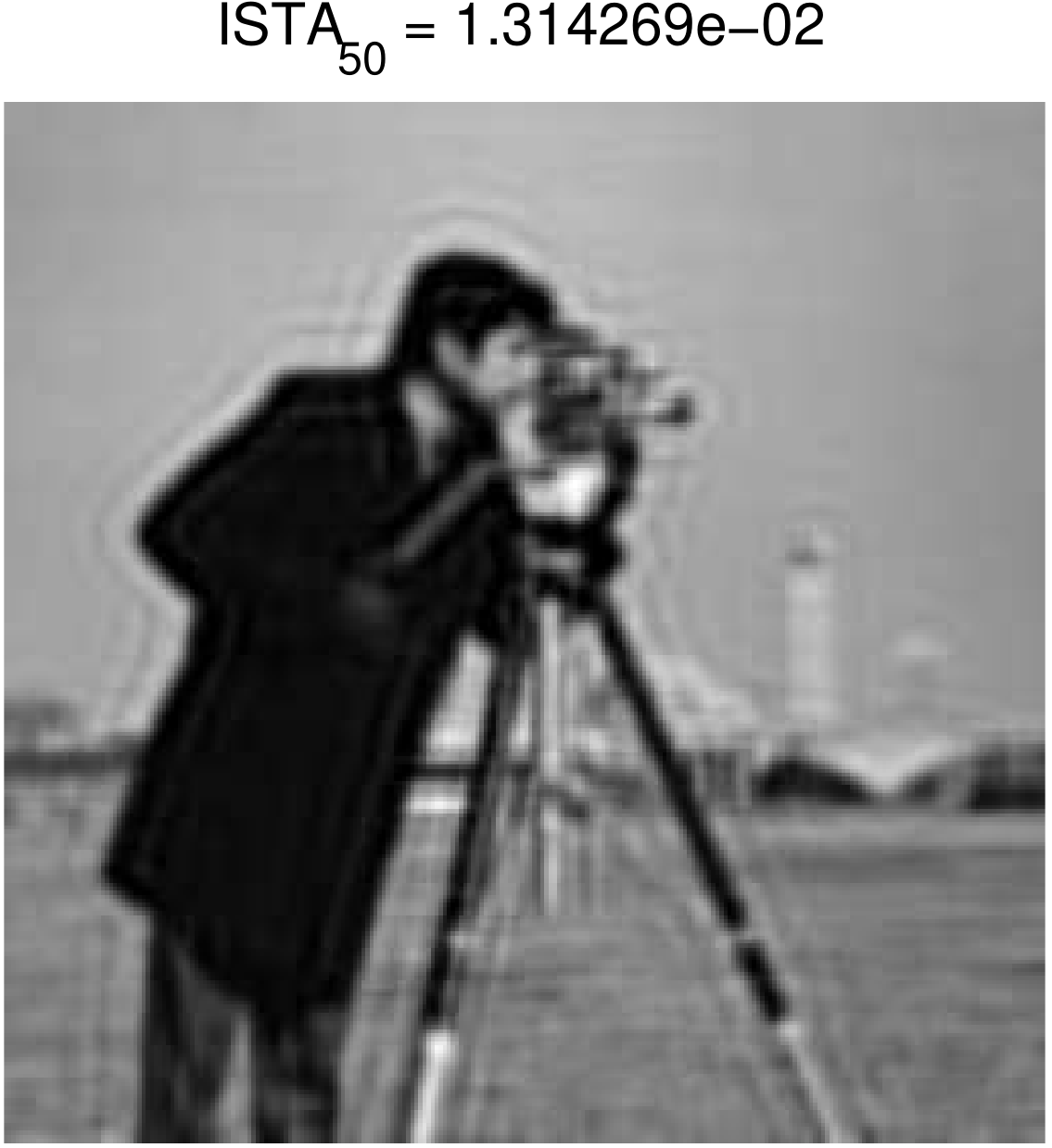}
	\includegraphics*[viewport= 143 249 470 606, width=0.32\textwidth]{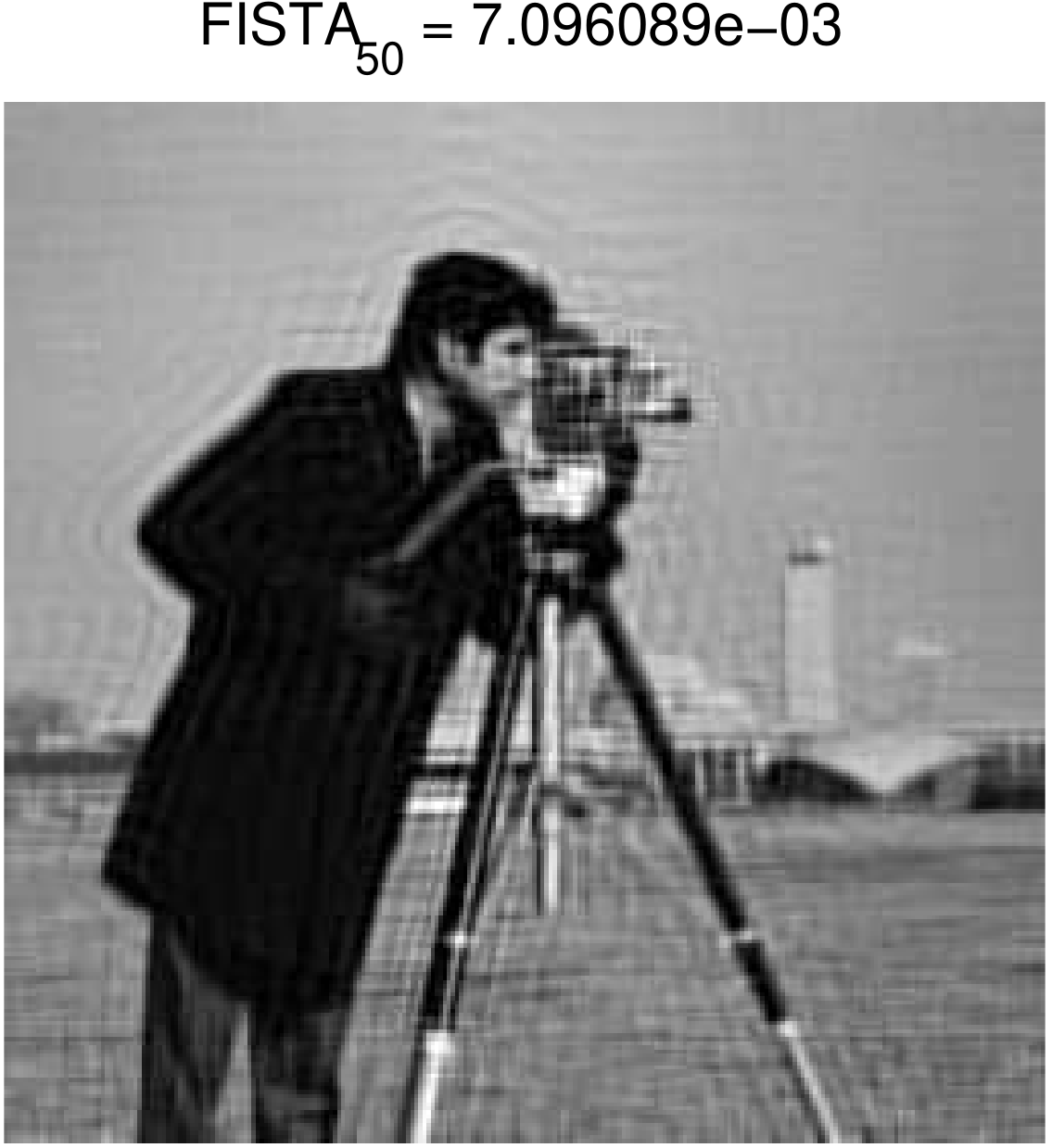}
	\includegraphics*[viewport= 143 249 470 606, width=0.32\textwidth]{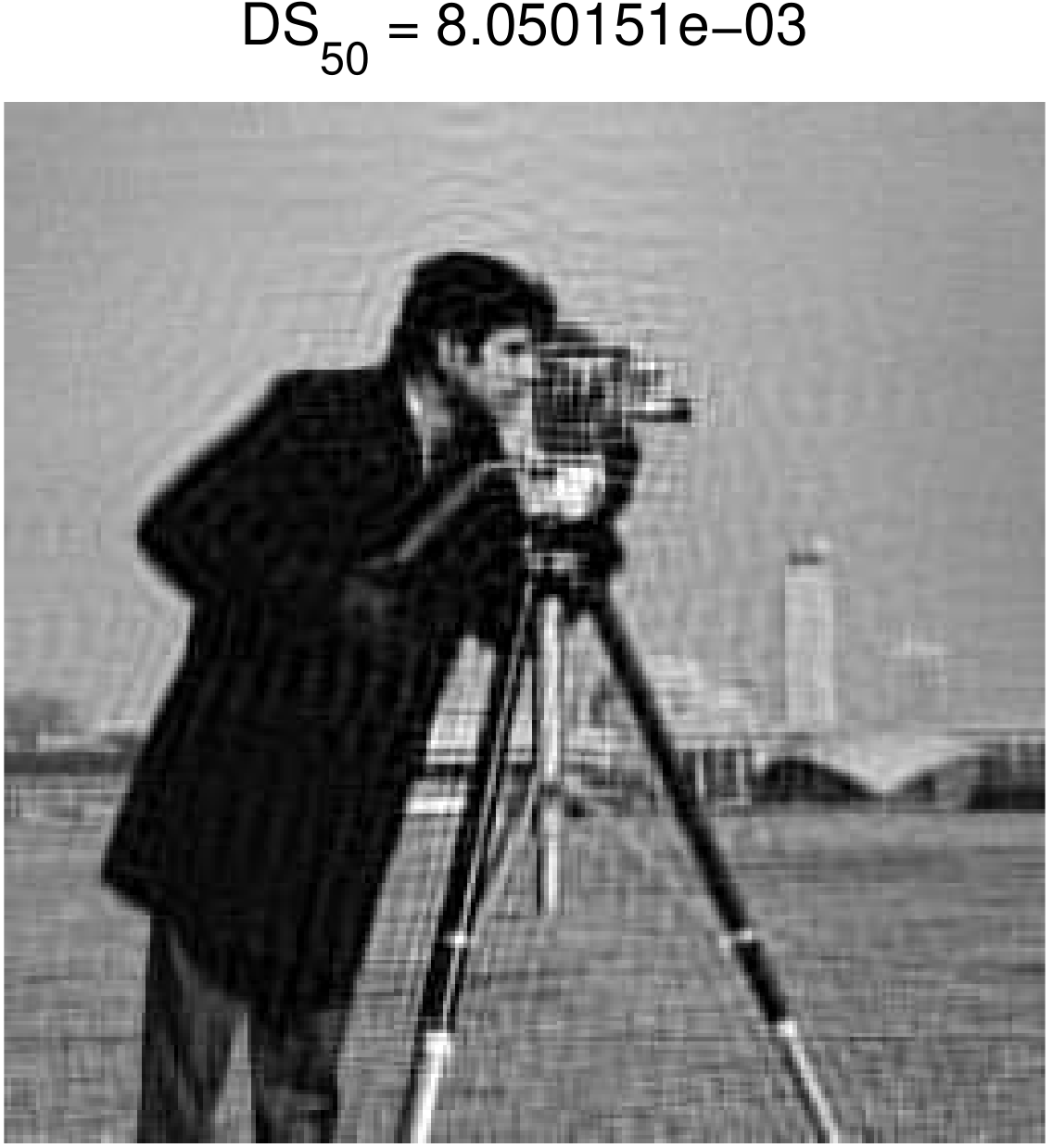}	
	\includegraphics*[viewport= 143 249 470 616, width=0.32\textwidth]{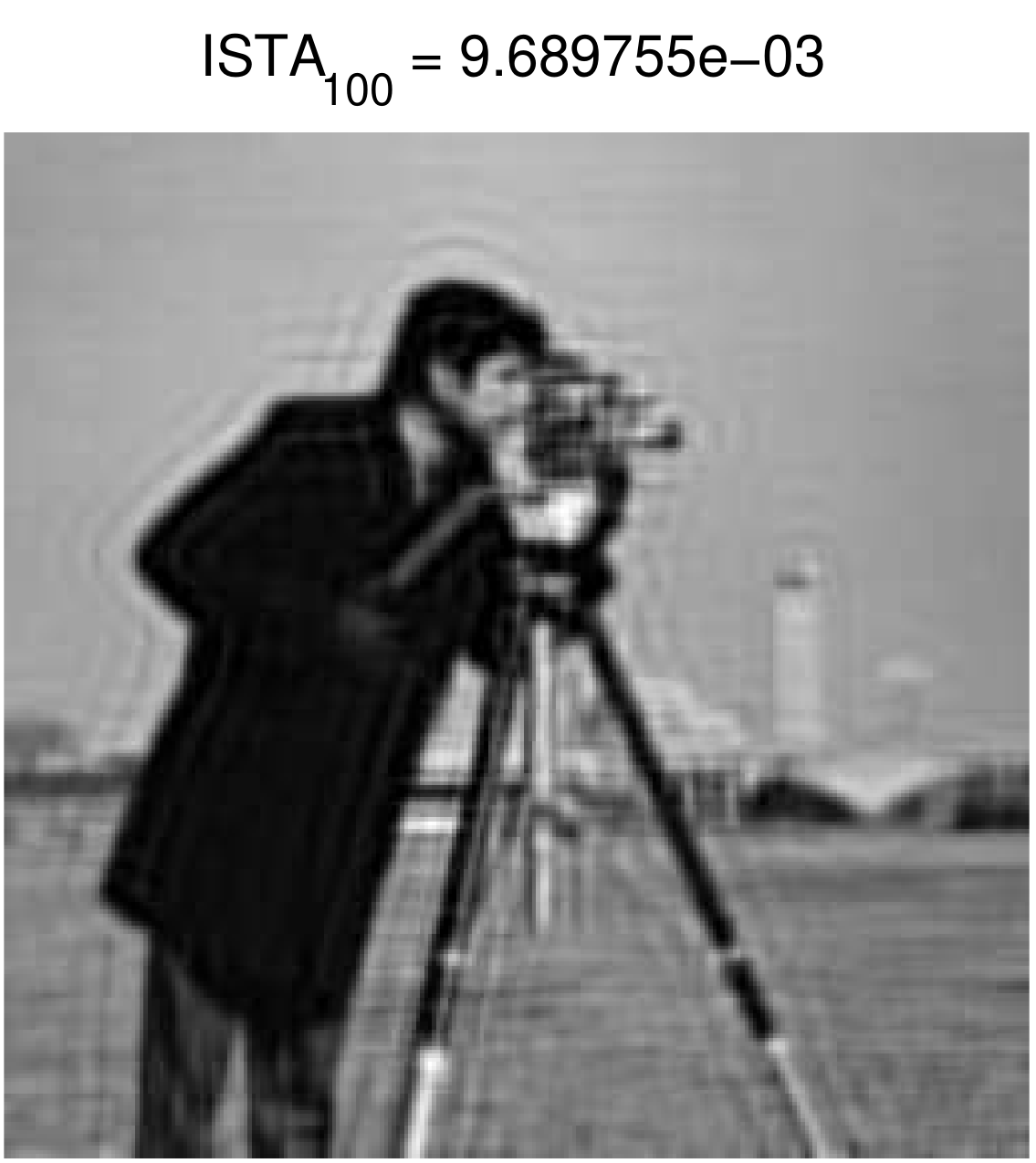}
	\includegraphics*[viewport= 143 249 470 616, width=0.32\textwidth]{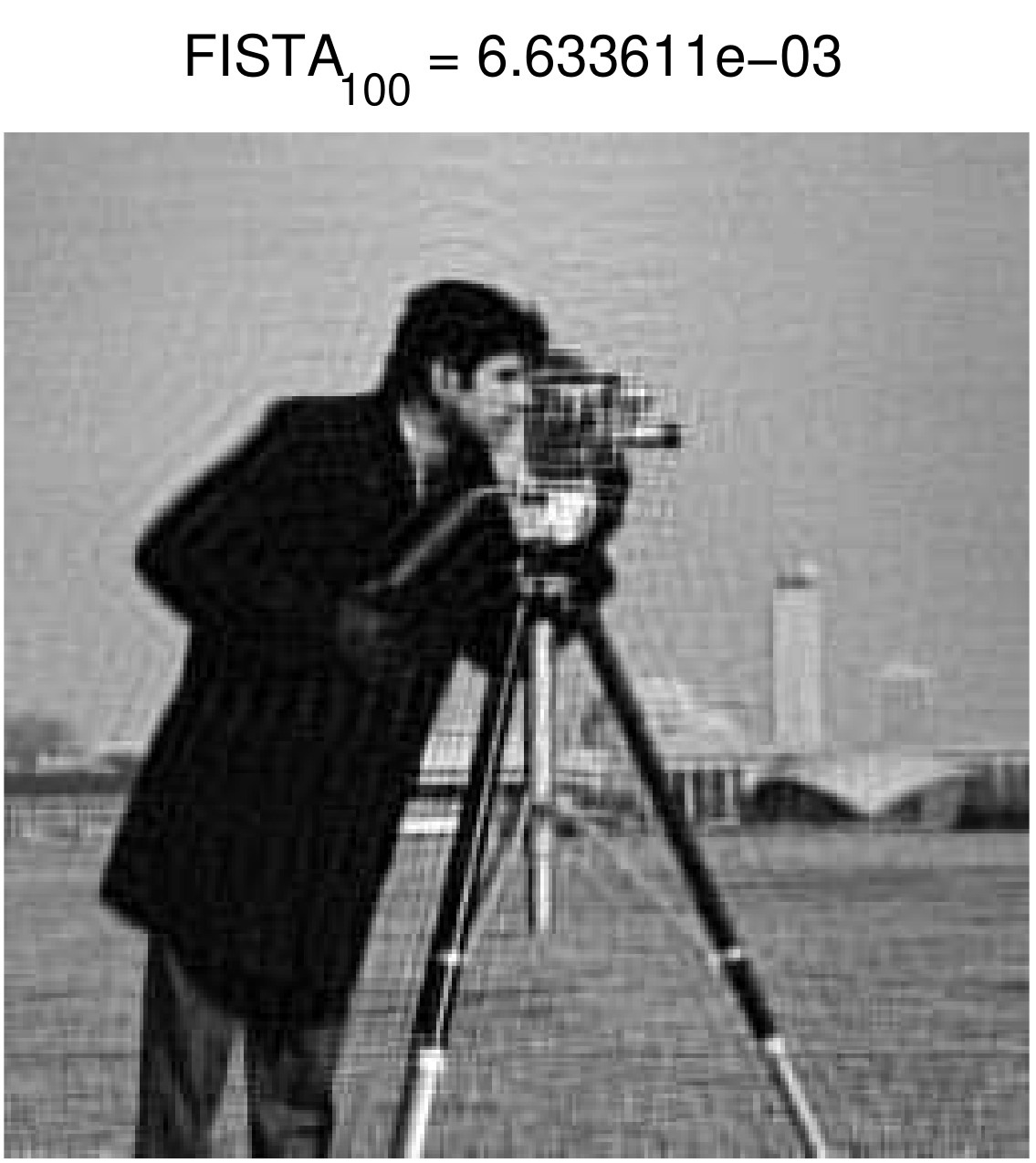}
	\includegraphics*[viewport= 143 249 470 616, width=0.32\textwidth]{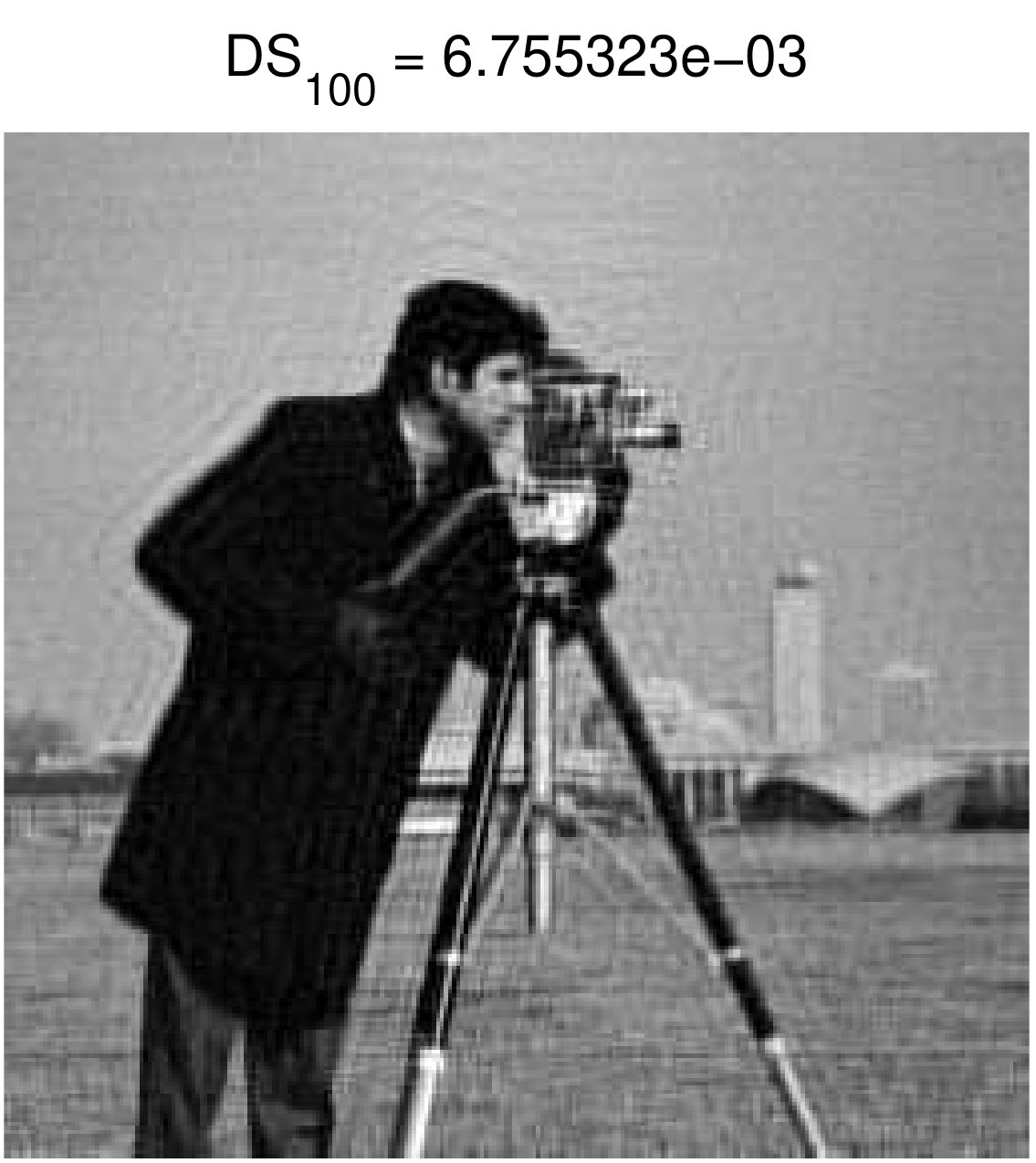}
	\caption{Iterations of ISTA, FISTA and double smoothing (DS) for solving $(P)$}
	\label{fig:cameramen_ISTA_FISTA_DS}	
\end{figure}

Figure \ref{fig:cameramen_ISTA_FISTA_DS} shows the iterations 50 and 100 of ISTA, FISTA and the double smoothing (DS) approach. The objective function values at iteration $k$ are denoted by ISTA$_k$, FISTA$_k$ and, respectively, DS$_k$ (e.\,g. DS$_k:=f(x_{f,p_k})+g(Ax_{f,p_k})$). All in all, the visual quality of the restored cameraman image after $100$ iterations, when using FISTA or DS, is quite comparable, whereas the recovered image by ISTA is still blurry. However, a valuable tool for measuring the quality of these images is the so-called \textit{improvement in signal-to-noise ratio (ISNR)}, which is defined as
$$ \text{ISNR}(k) = 10 \log_{10}\left( \frac{\left\|x-b\right\|^2}{\left\|x-x_k\right\|^2} \right) $$
where $x$, $b$ and $x_k$ denote the original, observed and estimated image at iteration $k$, respectively. Figure \ref{fig:l1_ISNR} shows the evolution of the ISNR values when using DS, FISTA and ISTA to solve $(P)$.
\begin{figure}[ht]	
	\centering
	\includegraphics*[viewport= 48 258 555 573, width=0.7\textwidth]{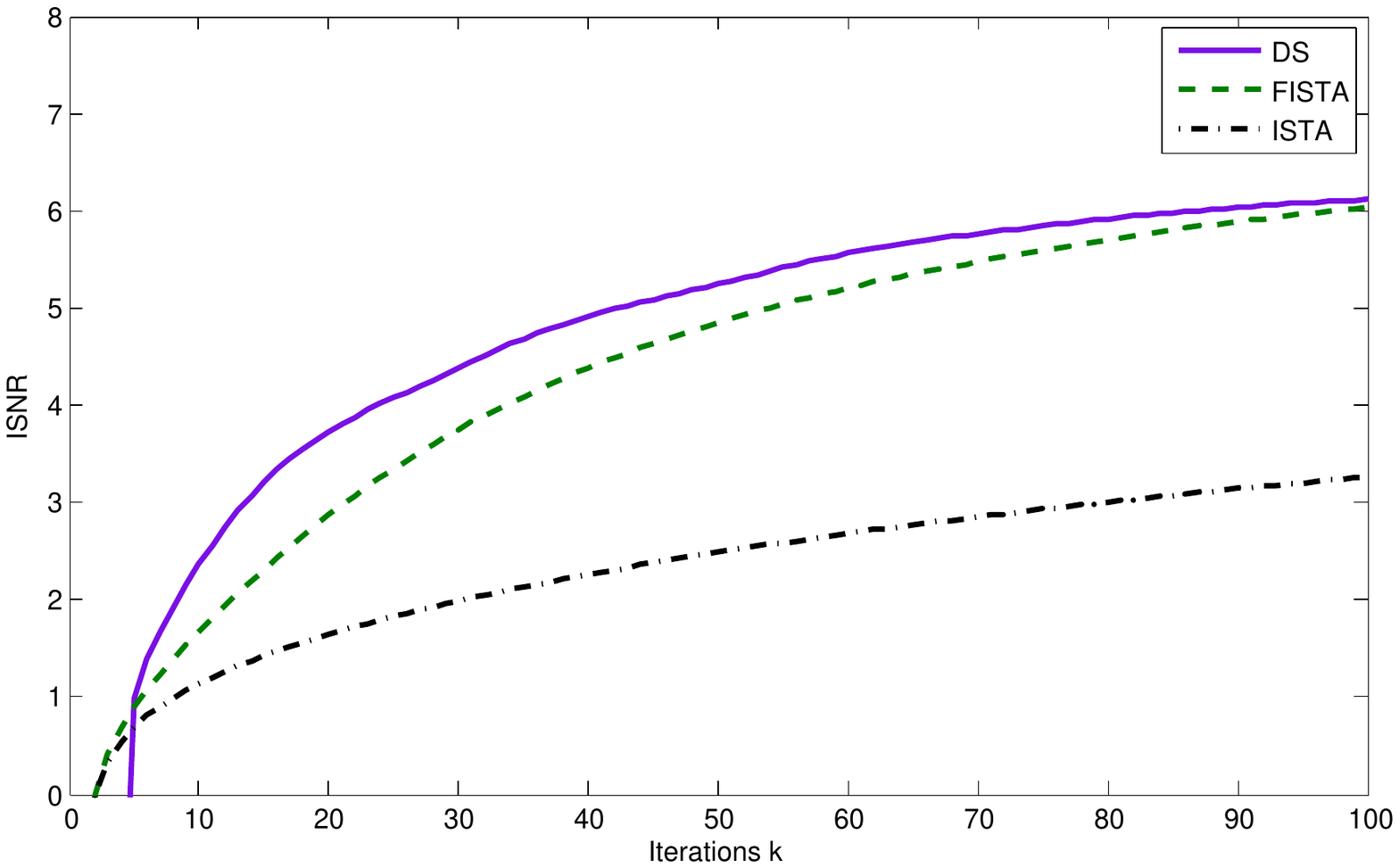}
	\caption{Improvement in signal-to-noise ratio (ISNR)}
	\label{fig:l1_ISNR}	
\end{figure}

\subsection{An \texorpdfstring{$l_2-l_1$}{l2-l1} regularization problem}\label{subsectionExample-l2l1}
The second convex optimization problem we solve is
\begin{align*}
	\hspace{-1.8cm}(P) \quad \quad \inf_{x \in S}{\left\{ \left\| Ax-b \right\|^2 + \lambda (\left\| x \right\|^2 + \left\| x \right\|_1) \right\}},
\end{align*}
where $S\subseteq \R^n$ is the $n$-dimensional cube $\left[0,1 \right]^n$ representing the pixel range, $\lambda > 0$ the regularization parameter and $\left\| \cdot \right\|^2 + \left\| \cdot \right\|_1$ the regularization functional, already used in \cite{BotHein11}. The problem to be solved can be equivalently written as
\begin{align*}
	\hspace{-1.8cm}(P) \quad \quad \inf_{x \in \mathbb{R}^n}{\left\{ f(x) + g(Ax)\right\}},
\end{align*}
for $f:\R^n \rightarrow \overline{\R}$, $f(x)=\lambda (\left\| x \right\|^2 + \left\| x \right\|_1) + \delta_{S}(x)$ and $g:\R^n \rightarrow \R$, $g(y)=\left\|y-b \right\|^2$. Thus $f$ is proper, $2\lambda$-strongly convex and lower semicontinuous with bounded domain and $g$ is a $2$-strongly convex function with full domain, differentiable everywhere and with Lipschitz continuous gradient having as Lipschitz constant $2$. This time we are in the setting of the Subsection \ref{subsectionSmoothSmoothStrongly}, the Lipschitz constant of the gradient of $\theta: \R^n \rightarrow \R$, $\theta(p)=f^*(A^*p)+g^*(-p)$, being $L= \frac{1}{2\lambda} + \frac{1}{2}$. By applying the double smoothing approach one obtains a rate of convergence of $O\left(\ln\left( \frac{1}{\epsilon}\right)\right)$ for solving $(P)$.

\begin{figure}[ht]
	\centering
	\includegraphics*[viewport= 91 322 541 500, width=0.8\textwidth]{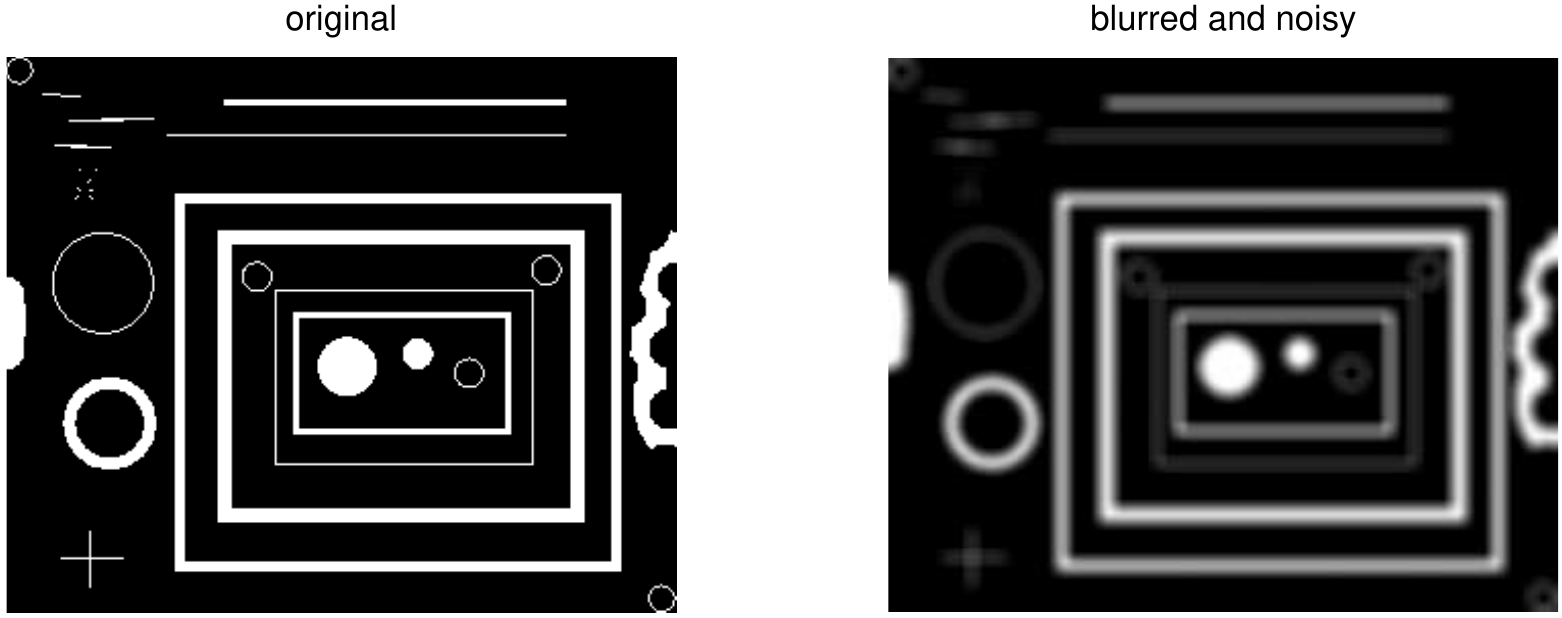}
	\caption{The $272 \times 329$ blobs test image}
	\label{fig:blobs}
\end{figure}

In this example we take a look at the \textit{blobs test image} shown in Figure \ref{fig:blobs} which is also part of the image processing toolbox in Matlab. The picture undergoes the same blur as described in the previous section. Since our pixel range has changed, we now use additive zero-mean white Gaussian noise with standard deviation $10^{-3}$ and the regularization parameter is changed to $\lambda=2$e-$5$.

We calculate next the sequences of approximate primal solutions $(x_{f,p_k})_{k \geq 0}$ and $(x_{g,p_k})_{k \geq 0}$. Indeed, for $k \geq 0$ we have
\begin{eqnarray*}
x_{f,p_k} & = & \argmin_{x\in \left[0,1\right]^n}{\left\{\lambda \left\| x \right\|^2 + \lambda \left\| x \right\|_1 - \left\langle A^*p_k,x \right\rangle \right\}}\\
& = & \argmin_{\substack{i=1,\dots,n \\ x_i \in \left[0,1 \right]}}{{\left\{  \sum_{i=1}^n{\left[ -(A^*p_k)_i x_i + \lambda x_i^2 + \lambda x_i \right]} \right\}}}
= \mathcal{P}_{\left[0,1 \right]^n}\left( \frac{1}{2\lambda}(A^*p_k - \lambda \mathbbm{1}^n) \right).
\end{eqnarray*}
and
\begin{align*}
x_{g,p_k} = \argmin_{x\in\R^n}{\left\{ \left\langle p_k,x \right\rangle + g(x) \right\}}
= \argmin_{x\in\R^n}{\left\{ \left\langle p_k,x \right\rangle + \left\|x - b\right\|^2 \right\}}
= b-\frac{1}{2}p_k.
\end{align*}

\begin{figure}[ht]	
	\centering
	\includegraphics*[viewport= 112 251 500 603, width=0.32\textwidth]{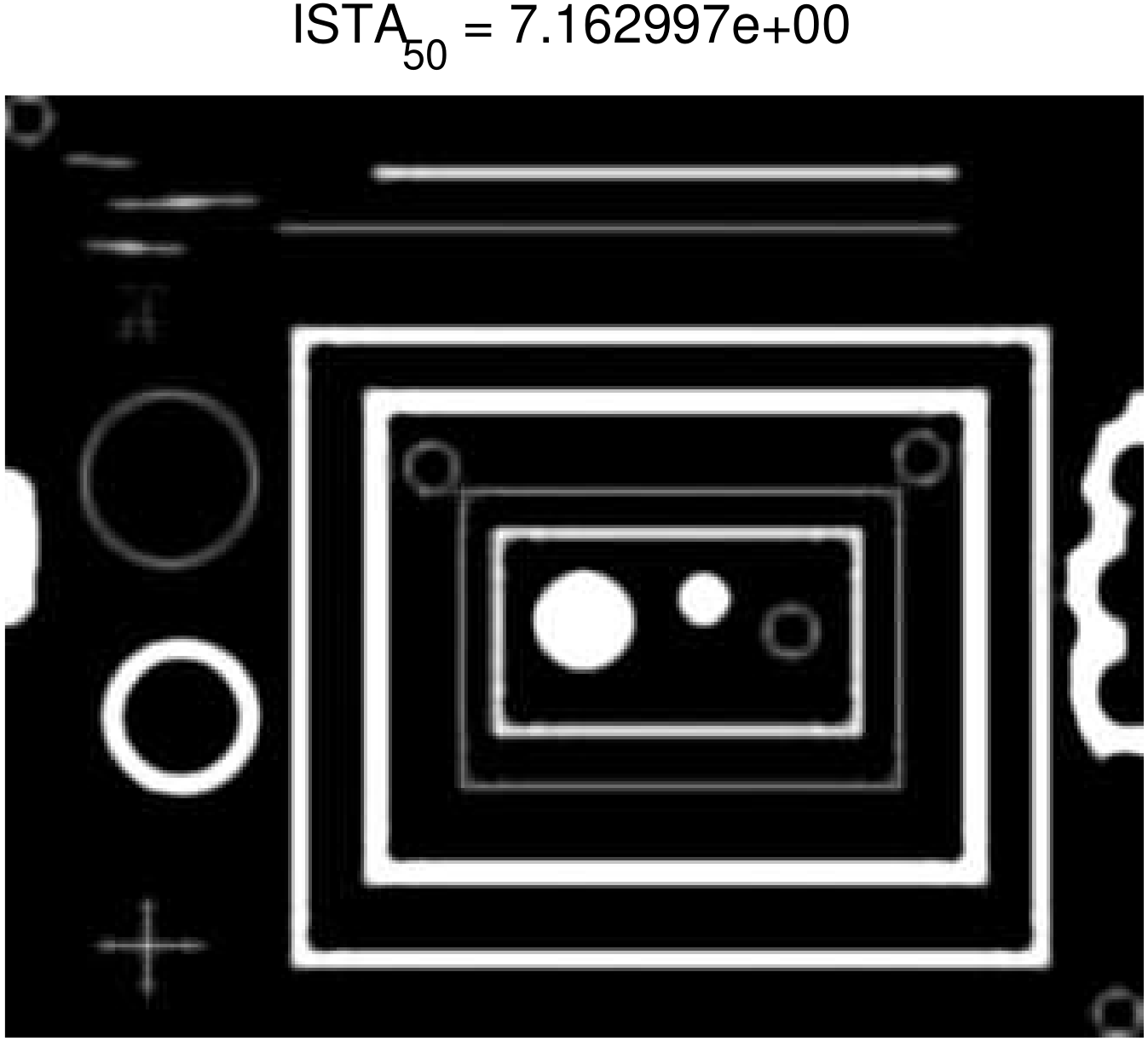}
	\includegraphics*[viewport= 112 251 500 603, width=0.32\textwidth]{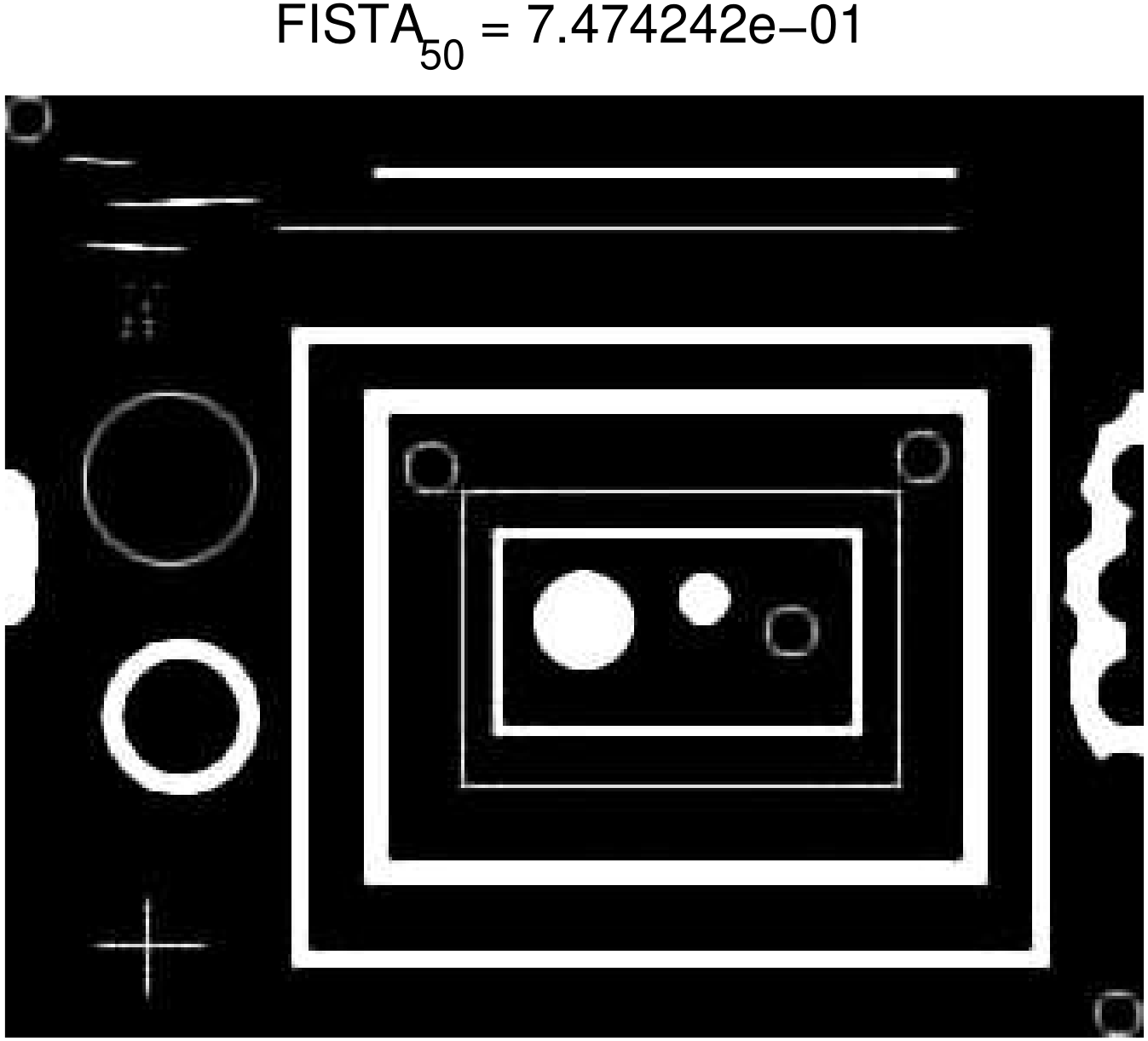}
	\includegraphics*[viewport= 112 251 500 603, width=0.32\textwidth]{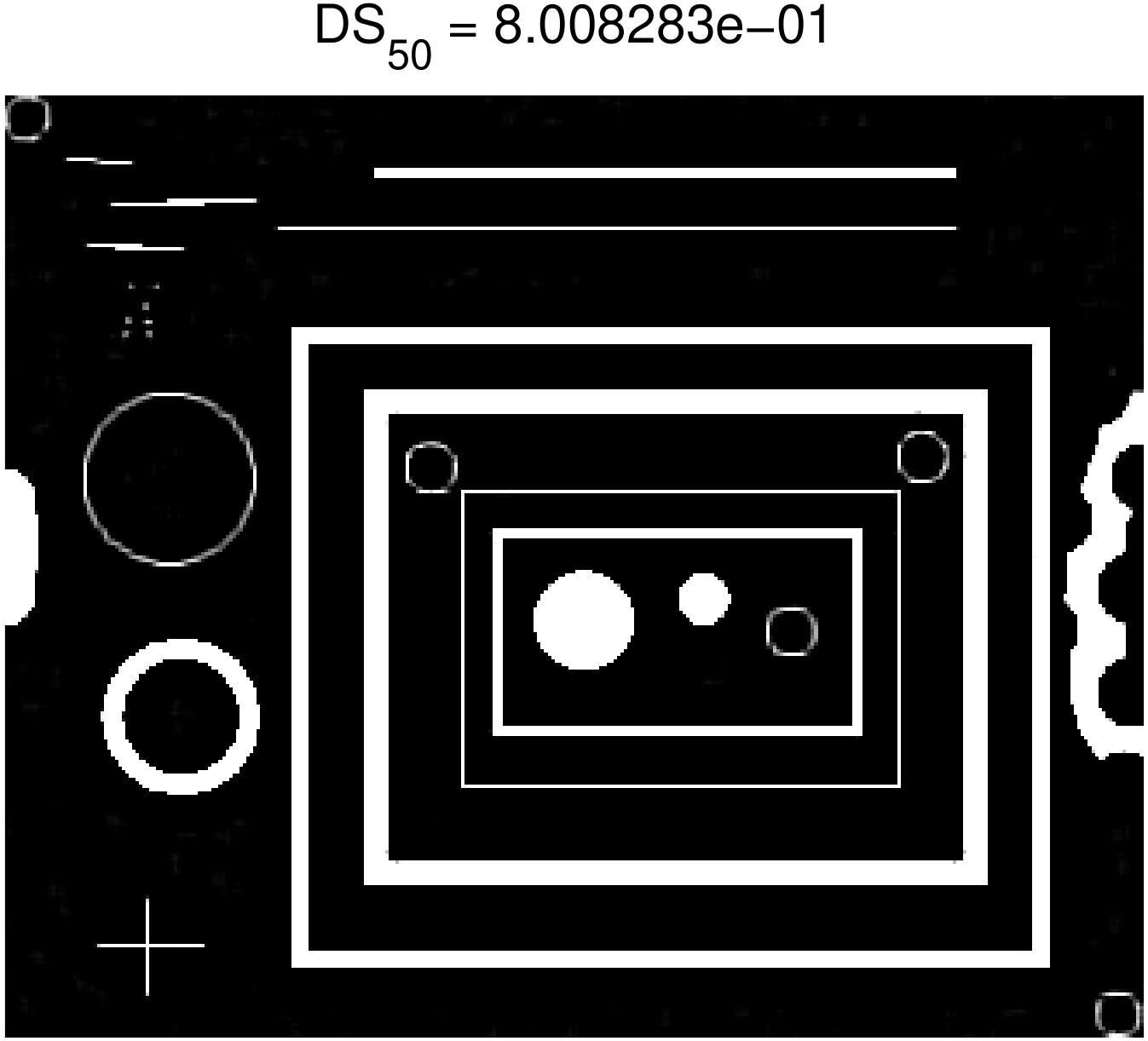}
	\includegraphics*[viewport= 112 251 500 613, width=0.32\textwidth]{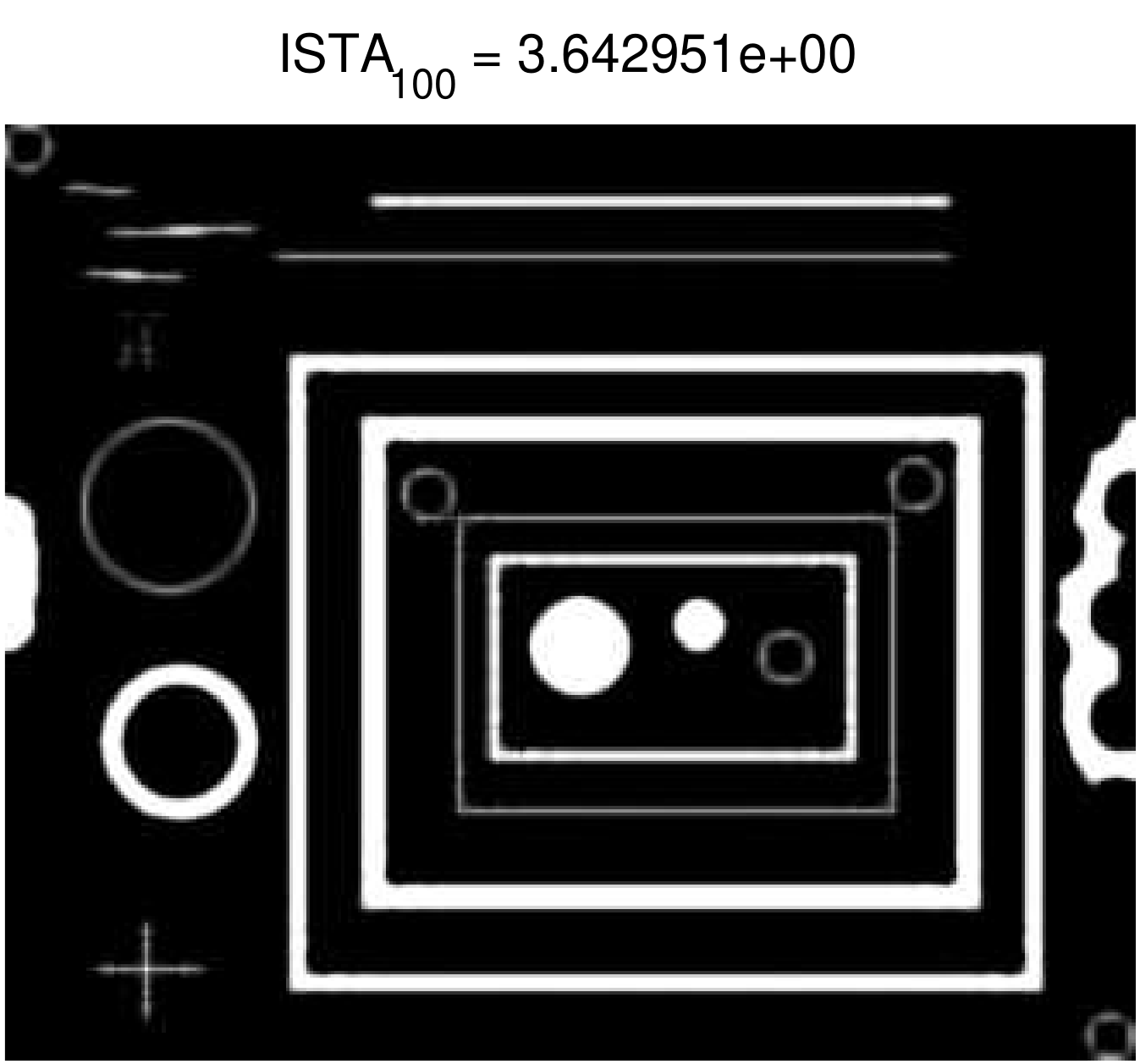}
	\includegraphics*[viewport= 112 251 500 613, width=0.32\textwidth]{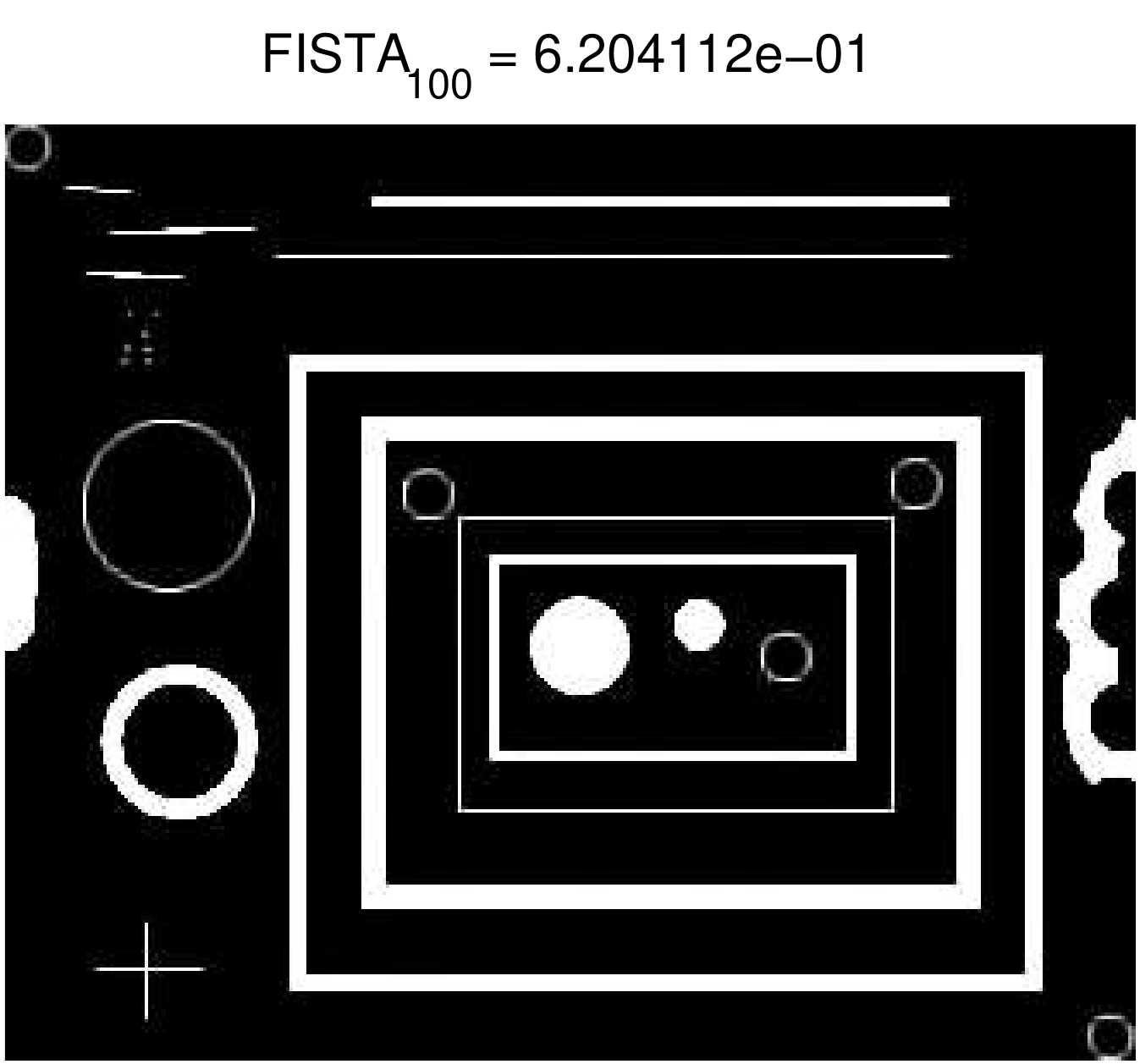}
	\includegraphics*[viewport= 112 251 500 613, width=0.32\textwidth]{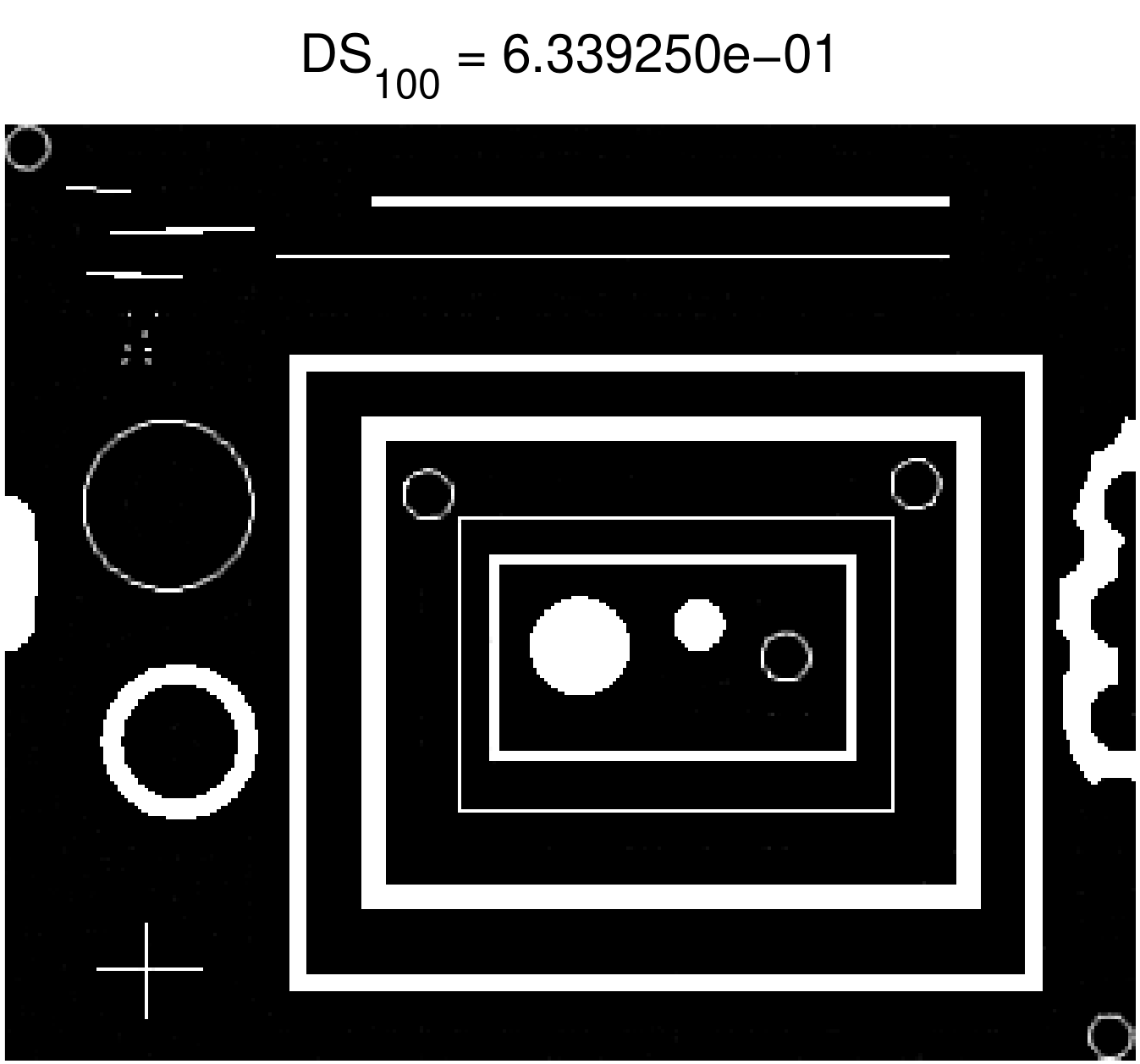}
	\caption{Iterations of ISTA, FISTA and double smoothing (DS) for solving $(P)$}
	\label{fig:blobs_ISTA_FISTA_DS}	
\end{figure}

Figure \ref{fig:blobs_ISTA_FISTA_DS} shows the iterations 50 and 100 of ISTA, FISTA and the double smoothing (DS) technique together with the corresponding function values denoted by ISTA$_k$, FISTA$_k$ or DS$_k$. As before, the function values of FISTA are slightly lower than those of DS, while ISTA is far behind these methods, not only from theoretical point of view, but also as it can be detected visually. Figure \ref{fig:l2_l1_ISNR} displays the improvement in signal-to-noise ration for ISTA, FISTA and DS and it shows that DS outperforms the other two methods from the point of view of the quality of the reconstruction.
\begin{figure}[ht]	
	\centering
	\includegraphics*[viewport= 44 258 555 573, width=0.7\textwidth]{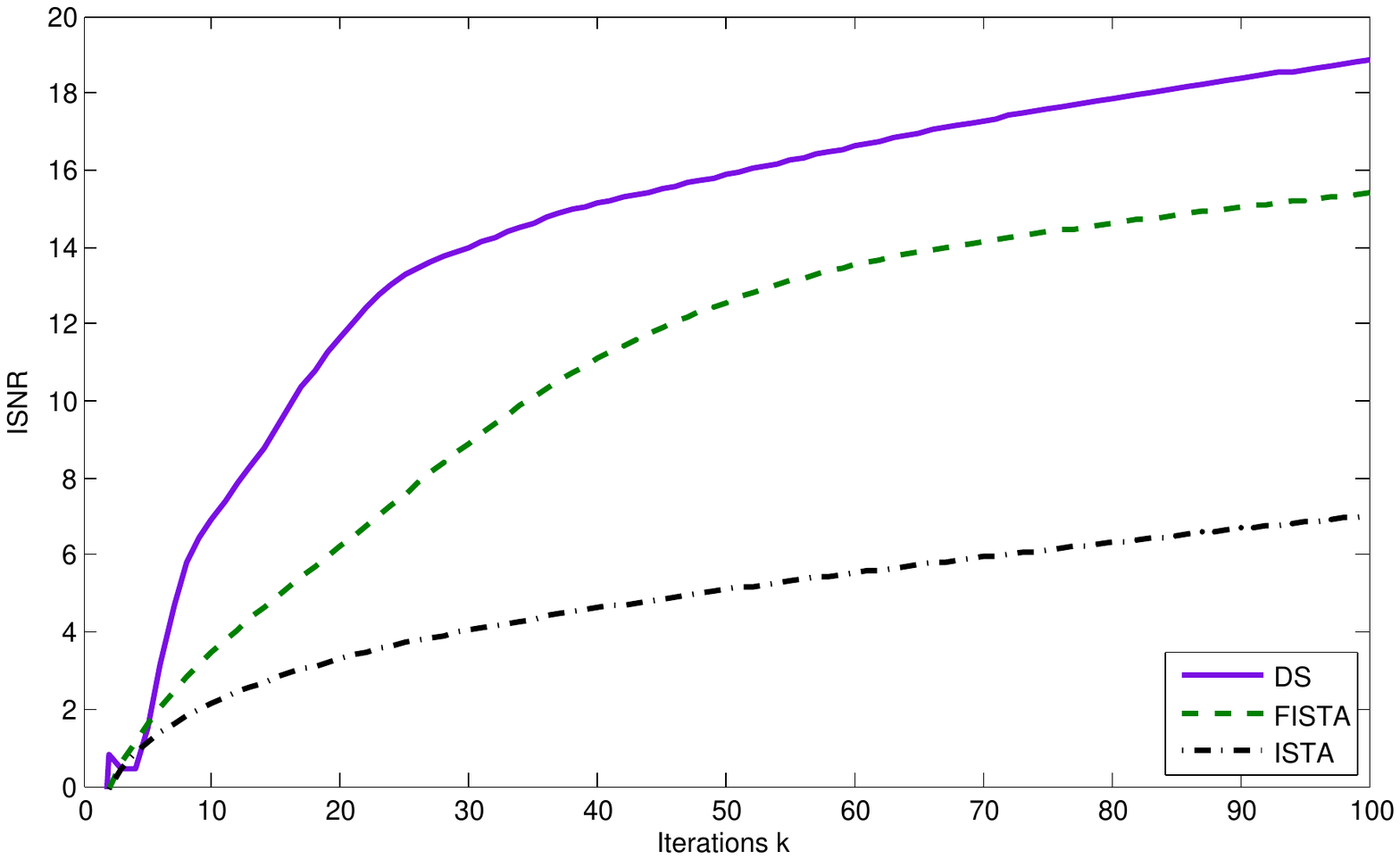}
	\caption{Improvement in signal-to-noise ratio (ISNR)}
	\label{fig:l2_l1_ISNR}	
\end{figure}

\section{Conclusions}

In this article we investigate the possibilities of accelerating the double smoothing technique when solving unconstrained nondifferentiable convex optimization problems. This method, which assumes the minimization of the doubly regularized Fenchel dual objective, allows in the most general case to reconstruct an approximately optimal primal solution in $O\left(\frac{1}{\epsilon} \ln\left( \frac{1}{\epsilon}\right)\right)$ iterations. We show that under some appropriate assumptions for the functions involved in the formulation of the problem to be solved this convergence rate can be improved to $O\left(\frac{1}{\sqrt{\epsilon}} \ln\left( \frac{1}{\epsilon}\right)\right)$, or even to $O\left(\ln\left( \frac{1}{\epsilon}\right)\right)$.

\end{document}